\documentclass{compositio}

\usepackage{amsmath, amssymb, amsthm}
\usepackage[colorlinks=true]{hyperref}
\usepackage[alphabetic]{amsrefs}
\usepackage{verbatim}
\usepackage{xcolor}
\usepackage{graphicx}
\usepackage{url}
\usepackage{enumitem}

\newtheorem{lemma}{Lemma}[section]
\newtheorem{theorem}[lemma]{Theorem}
\newtheorem{corollary}[lemma]{Corollary}
\newtheorem{proposition}[lemma]{Proposition}
\theoremstyle{definition}
\newtheorem{definition}[lemma]{Definition}

\newtheorem{example*}[lemma]{Example}
\newtheorem{remark}[lemma]{Remark}

\theoremstyle{remark}

\makeatletter
\newtheorem*{rep@theorem}{\rep@title}
\newcommand{\newreptheorem}[2]{%
\newenvironment{rep#1}[1]{%
\def\rep@title{{\bf #2 \ref{##1}}}%
\begin{rep@theorem}}%
{\end{rep@theorem}}}
\makeatother
\newreptheorem{theorem}{Theorem}

\DeclareRobustCommand{\qedify}[1]{%
  \ifmmode \quad\hbox{#1}
  \else
    \leavevmode\unskip\penalty9999 \hbox{}\nobreak\hfill
    \quad\hbox{#1}%
  \fi
}
\newenvironment{example}{\begin{example*}\pushQED{\qedify{$\diamondsuit$}}}
{\popQED\end{example*}}

\newcommand{\R}{R}
\newcommand{\K}{K}

\newcommand{\Rbar}{\ensuremath{\overline{\mathbb R}}}

\newcommand{\Sx}{\Rbar[x_0,\dotsc,x_n]}
\newcommand{\Bx}{\mathbb{B}[x_0,\dotsc,x_n]}
\newcommand{\Kx}{\K[x_0,\dotsc,x_n]}
\newcommand{\M}{\mathcal{M}}

\newcommand{\C}{\mathcal{C}}
\newcommand{\V}{\mathcal{V}}
\newcommand{\B}{\mathcal{B}}

\DeclareMathOperator{\Hom}{Hom}
\DeclareMathOperator{\inn}{in}

\DeclareMathOperator{\rank}{rank}
\DeclareMathOperator{\Sxpec}{Spec}
\DeclareMathOperator{\val}{val}
\DeclareMathOperator{\im}{im}
\DeclareMathOperator{\trop}{trop}
\DeclareMathOperator{\supp}{supp}

\DeclareMathOperator{\spann}{span}

\DeclareMathOperator{\GL}{GL}

\DeclareMathOperator{\mon}{Mon}

\DeclareMathOperator{\Proj}{Proj}

\DeclareMathOperator{\Cox}{Cox}
\DeclareMathOperator{\coker}{coker}

\DeclareMathOperator{\cone}{cone}
\DeclareMathOperator{\TSxpec}{TSpec}

\makeatletter
\newcommand{\superimpose}[2]
{{\ooalign{$#1\@firstoftwo#2$\cr\hfil$#1\@secondoftwo#2$\hfil\cr}}}
\makeatother
\newcommand{\ttimes}{\hspace{0.4mm}{\mathpalette\superimpose{{\circ}
{\cdot}}}\hspace{0.4mm}}
\newcommand{\tplus}{\mathrel{\oplus}}

\begin{document}

\title{Tropical Ideals}

\author{Diane Maclagan}
\email{D.Maclagan@warwick.ac.uk}
\address{Mathematics Institute, University of Warwick, Coventry CV4 7AL, United 
Kingdom.}

\author{Felipe Rinc\'on}
\email{feliperi@math.uio.no}
\address{Department of Mathematics, University of Oslo, 0851 Oslo, Norway.}


\begin{abstract}
We introduce and study a special class of ideals, called tropical
ideals, in the semiring of tropical polynomials, with the goal of
developing a useful and solid algebraic foundation for tropical
geometry. The class of tropical ideals strictly includes the tropicalizations
of classical ideals, and allows us to define subschemes of tropical toric varieties,
generalizing \cite{Giansiracusa2}.  We investigate some
of the basic structure of tropical ideals, and show that they satisfy many desirable
properties that mimic the classical setup. In particular, every
tropical ideal has an associated variety, which we prove is
always a finite polyhedral complex.  In addition we show that tropical
ideals satisfy the ascending chain condition, even though they are
typically not finitely generated, and also the weak Nullstellensatz.
\end{abstract}

\maketitle

\section{Introduction}

Tropical algebraic geometry is a piecewise linear shadow of algebraic geometry,
in which varieties are replaced by polyhedral complexes.  
This area has grown significantly in the past decade 
and has had great success in numerous applications,
such as  Mikhalkin's calculation of Gromov-Witten invariants of
$\mathbb P^2$ \cite{MikhalkinP2}, the work of
Cools-Draisma-Payne-Robeva \cite{CDPR} and Jensen-Payne
\cites{JensenPayne1, JensenPayne2} on Brill-Noether theory, and the
Gross-Siebert program in mirror symmetry \cite{GrossBook}.

One current limitation of the theory, however, is that most techniques
developed to date are focused on tropical varieties and tropical
cycles, as opposed to schemes or more general spaces.  Many of the
standard tools of modern algebraic geometry thus do not yet have a
tropical counterpart.  

In \cite{Giansiracusa2}, Jeffrey and
Noah Giansiracusa described how to tropicalize a subscheme of a toric
variety using congruences on the semiring of tropical polynomials.
The authors of this paper developed this further in
\cite{MaclaganRincon1}, clarifying the connection to 
tropical linear spaces and valuated matroids.

Building on this work, in this paper we investigate a special class of
ideals in the semiring of tropical polynomials, called tropical
ideals, in which bounded-degree pieces are ``matroidal''.  This
allows us to define tropical subschemes of a tropical toric variety,
which include, but are not limited to, tropicalizations of classical
subschemes of toric varieties.  We show that, unlike for more general
ideals, varieties of tropical ideals are always finite polyhedral complexes.
In addition, even though the semiring of tropical polynomials is far
from Noetherian, the restricted class of tropical ideals satisfies the
ascending chain condition.  They also satisfy a version of the
Nullstellensatz that is completely analogous to the classical formulation.

We denote by $\Rbar$ the tropical semiring $\mathbb R \cup \{\infty\}$
with the operations tropical sum $\tplus \,\, = \min$ and tropical
multiplication $\ttimes = +$.  The semiring of tropical polynomials
$\Rbar[x_1,\dots,x_n]$ consists of polynomials with coefficients in
$\Rbar$ where all operations are tropical.  For simplicity, in this
introduction we describe our results in the case where the ambient toric variety 
is tropical affine space. 
A general treatment for any tropical toric variety is given in Section
\ref{s:4}.

\begin{definition}\label{d:tropicalideal}
An ideal $I \subseteq \Rbar[x_1,\dots,x_n]$ is a 
{\em tropical ideal} if for each degree $d \geq 0$ the set $I_{\leq d}$ of polynomials
in $I$ of degree at most $d$ is a tropical linear space, or equivalently, $I_{\leq d}$ is 
the set of vectors of a valuated matroid. 
More concretely, $I$ is a tropical ideal if it satisfies the following
``monomial elimination axiom'': 

\parbox[center][][c]{0.95\textwidth}
{
$\bullet$ For any $f, g \in I_{\leq d}$ and any monomial 
$\mathbf{x}^{\mathbf{u}}$ for which 
$[f]_{\mathbf{x}^{\mathbf{u}}} = [g]_{\mathbf{x}^{\mathbf{u}}} \neq \infty$,
there exists $h \in I_{\leq d}$ such that
$[h]_{\mathbf{x}^{\mathbf{u}}} = \infty$ and $[h]_{\mathbf{x}^{\mathbf{v}}} 
\geq \min( [f]_{\mathbf{x}^{\mathbf{v}}}, [g]_{\mathbf{x}^{\mathbf{v}}})$ 
for all monomials $\mathbf{x}^{\mathbf{v}}$, with the equality holding
whenever $[f]_{\mathbf{x}^{\mathbf{v}}} \neq [g]_{\mathbf{x}^{\mathbf{v}}}$.
}

\noindent Here we use the notation $[f]_{\mathbf{x}^{\mathbf{u}}}$ to denote
the coefficient of the monomial $\mathbf{x}^{\mathbf{u}}$ in the tropical 
polynomial $f$.
\end{definition}

If $J \subseteq K[x_1,\dotsc,x_n]$ is an ideal
then its tropicalization $\trop(J) \subseteq \Rbar[x_1,\dots,x_n]$ 
is a tropical ideal. 
However, the class of tropical ideals is larger;
we exhibit in Example~\ref{ex:non-realizable} a tropical ideal 
that cannot be realized as $\trop(J)$ 
for any ideal $J \subseteq K[x_1,\dotsc,x_n]$ over any field $K$.

As we describe below, 
the ``monomial elimination axiom'' required for tropical ideals
makes up for the lack of subtraction in the tropical semiring,
and gives tropical ideals a rich algebraic structure 
reminiscent of ideals in a polynomial ring. In particular, tropical ideals
seem to be better suited than general ideals of $\Rbar[x_1,\dots,x_n]$ for 
studying the underlying geometry; see Remark \ref{rem:tropicalspectra}.

Given a polynomial $f \in \Rbar[x_1,\dots,x_n]$, its associated {\em variety} is
\[
V(f) := \{ \mathbf{w} \in \Rbar^{n} : f(\mathbf{w}) = \infty \text{ or the minimum in }
f(\mathbf{w}) \text{ is achieved at least twice}\}. 
\]
If $I \subseteq \Rbar[x_1,\dots,x_n]$ is an ideal, the {\em variety} of $I$ is 
\begin{equation}  \label{eqtn:variety}
V(I) := \bigcap_{f \in I} V(f).
\end{equation}
If $I$ is an arbitrary ideal in $\Rbar[x_1,\dots,x_n]$ then
$V(I)$ can be a fairly arbitrary subset of $\Rbar^{n}$;
see Example~\ref{e:arbitrary}. 
In particular, $V(I)$ might not even be polyhedral.  
However, if $I$ is a tropical
ideal, our main result shows that this is not the case:

\begin{theorem}\label{t:polyhedralcomplexintro}
If $I \subseteq \Rbar[x_1,\dots,x_n]$ is a tropical ideal then the
variety $V(I) \subseteq \Rbar^{n}$ is a finite 
polyhedral complex.
\end{theorem}

Our proof of Theorem~\ref{t:polyhedralcomplexintro} generalizes
the case where $I = \trop(J)$ for a classical ideal $J$:
we develop a Gr\"obner theory for tropical ideals,
and show that any tropical ideal has a finite 
Gr\"obner complex, as in \cite{TropicalBook}*{\S 2.5}. The variety of
$I$ is then a subcomplex of its Gr\"obner complex.

A {\em tropical basis} for a tropical ideal $I$ is a collection of
polynomials in $I$ the intersection of whose varieties is the variety
$V(I)$. It is well known that a tropical ideal of the form $\trop(J)$
always admits a finite tropical basis; we show in Theorem
\ref{t:TropicalBasis} that this is in fact true for any tropical
ideal.

We also investigate commutative algebraic properties of tropical ideals.
The fact that tropical ideals are ``matroidal'' allows us to naturally define
the Hilbert function of any homogeneous tropical ideal. In the case where
$I = \trop(J)$ for a classical ideal $J$, the Hilbert function of $I$
agrees with the Hilbert function of $J$. In Proposition 
\ref{prop:eventuallypoly} 
we show that, just as in the classical case, the Hilbert function
of any homogeneous tropical ideal is eventually polynomial.

The semiring $\Rbar[x_1,\dots,x_n]$ is not Noetherian, and tropical
ideals are almost never finitely generated; see Example 
\ref{e:notFinitelyGenerated}. 
Moreover, Example~\ref{e:infinitefamily} 
shows an infinite family of distinct homogeneous tropical ideals $\{ I^j \}_{j \geq 1}$,
all of which have the same Hilbert function, such that for any $d \geq 0$, 
if $k, l \geq d$, then the tropical ideals
$I^k$ and $I^l$ agree on all their graded pieces of degree at most $d$.
Nonetheless, tropical ideals do satisfy the following Noetherian property.

\begin{theorem}\label{t:acc}
{(Ascending chain condition.)} There is no infinite ascending chain
$I_1 \subsetneq I_2 \subsetneq I_3 \subsetneq \dotsb$ of tropical ideals in 
$\Rbar[x_1,\dots,x_n]$.
\end{theorem}

There are several versions of the Nullstellensatz for tropical geometry 
in the literature; see for example 
\cites{ShustinIzhakian, BertramEaston, JooMincheva, GrigorievPodolskii}. 
Most of these concern arbitrary finitely generated ideals in $\Rbar[x_1,\dots,x_n]$.
In our case, the rich structure we impose on tropical ideals allows us
to use the results in \cite{GrigorievPodolskii} to get the following
familiar formulation.

\begin{theorem}\label{t:tropicalnullstellensatzintro}(Tropical Nullstellensatz.)
If $I \subseteq \Rbar[x_1,\dots,x_n]$ is a tropical ideal then 
the variety $V(I) \subseteq \Rbar^n$ is empty if and only if $I$ 
is the unit ideal $\langle 0 \rangle$.
\end{theorem}

Several of our results for tropical ideals imply the analogous versions
for classical ideals, simply by considering tropical ideals of the
form $\trop(J)$ for $J \subseteq \K[x_1,\dots,x_n]$. Since our
arguments for tropical ideals are all of combinatorial nature, our
approach has the appealing feature of providing completely
combinatorial proofs for some important well-known statements, such as
the existence of Gr\"obner complexes for classical homogeneous ideals and 
the existence of finite tropical bases for classical ideals. 
This suggests that more of standard commutative
algebra can be tropicalized and extended in this manner.

We conclude with a brief description of how the paper is organized.
The basics of valuated matroids and tropical ideals are explained in
Section~\ref{s:2}. Some Gr\"obner theory for tropical ideals is
developed in Section~\ref{s:3}, together with a discussion on Hilbert
functions and a proof of Theorem~\ref{t:acc}
(Theorem~\ref{p:ascendingchain}). Tropical subschemes of tropical toric
varieties are introduced in Section~\ref{s:4}.  Finally, in
Section~\ref{s:5} we prove Theorem~\ref{t:polyhedralcomplexintro}
(Theorem~\ref{t:polyhedralcomplex}), as well as the existence of finite
tropical bases (Theorem~\ref{t:TropicalBasis}) and Theorem
\ref{t:tropicalnullstellensatzintro}
(Theorem~\ref{t:tropicalnullstellensatz}).

\noindent {\bf Acknowledgments}.  We thank Jeff Giansiracusa for
several useful conversations, and Paolo Tripoli for insight into
Example~\ref{e:arbitrary}. The first author was partially supported by
EPSRC grant EP/I008071/1 and a Leverhulme Trust International Academic Fellowship.  The second author was supported by the
Research Council of Norway grant 239968/F20.

\section{Tropical Ideals}

\label{s:2}

In this section we introduce  tropical ideals, together with several 
examples. 
We first recall some basics of valuated matroids.  

Throughout this paper we take $(\R, \tplus, \ttimes)$ to be a
semifield, with the following extra properties.  For compatibility
with tropical notation we write $\infty$ for the additive identity of
$\R$, and $\R^*$ for $\R \setminus \{ \infty \}$.  We require that $(\R^*,
\ttimes)$ is a totally ordered group, and that the
addition $\tplus$ satisfies $a \tplus b = \min(a,b)$.  Under these hypotheses,
$R$ is sometimes called a valuative semifield. We extend the
total ordering on the multiplicative group to all of $\R$ by making
$\infty$ the largest element.  

The main example for us will be the
tropical semiring (or min-plus algebra) $\Rbar$:
\[
\Rbar := (\mathbb R \cup \{ \infty \}, \tplus, \ttimes), \text{ where }
\tplus \,\, := \min \,\, \text{ and } \,\, \ttimes := +.
\]
Here the ordering on the
multiplicative group $\mathbb R$ is the standard one.  Another
important example is the Boolean subsemiring $\mathbb B$ of $\Rbar$:
$$\mathbb B := (\{ 0, \infty \}, \tplus, \ttimes).$$ 

An ideal in a semiring $R$ is a nonempty subset of $R$ closed under addition and under
multiplication by elements of $R$.

We denote by 
$\R[x_0, \dots,x_n]$
the semiring of polynomials in the variables
$x_0,\dots,x_n$ with coefficients in $\R$. 
Note that in the case that $R = \Rbar$, 
elements of $\Rbar[x_0,\dots,x_n]$ 
are regarded as 
polynomials and not functions; for example, the 
polynomials
$f(x) = x^2 \tplus 0$ and $g(x)= x^2 \tplus 1 \ttimes x \tplus 0$ are
distinct, even though $f(w)=g(w)$ for all $w \in \mathbb R$.  
The {\em support} of a 
polynomial $f = \bigoplus a_{\mathbf{u}}
\ttimes \mathbf{x}^{\mathbf{u}}$ is
\[
\supp(f) := \{\mathbf{u} \in \mathbb N^{n+1}: a_{\mathbf{u}} \neq \infty\}. 
\]
We call $a_{\mathbf{u}}$ the coefficient in $f$ of the monomial 
$\mathbf{x}^{\mathbf{u}}$. 

\subsection{Valuated matroids}

Valuated matroids are a generalization of the notion of matroids, 
introduced by Dress and Wenzel in \cite{DressWenzel}. 
We recall some of the necessary background on valuated matroids and 
tropical linear spaces; for basics of standard matroids, see, 
for example, \cite{Oxley}.

Let $E$ be a finite set, and let $r \in \mathbb N$.  Denote by
$\binom{E}{r}$ the collection of subsets of $E$ of size $r$.  A {\em
  valuated matroid} on the ground set $E$ with values in the
semifield $R$ is a pair $\M = (E,p)$ where $p : {\binom{E}{r}}
\rightarrow R$ satisfies the following properties:
\begin{itemize}
\item There exists $B \in \binom{E}{r}$ such that $p(B) \neq \infty$. 
\item {\em Valuated basis exchange axiom:} For every $A, B \in \binom{E}{r}$ and 
every 
$a \in A \setminus B$ there exists $b \in B \setminus A$ with 
\begin{equation}
\label{eqtn:tropicalplucker}
p(A) \ttimes p(B) \geq p(A \cup b - a)\ttimes  p(B \cup a - b).
\end{equation}
\end{itemize}
In the case that $R = \Rbar$, the valuated basis exchange axiom is
equivalent to the tropical Pl\"ucker relations; see, for instance,
\cite{TropicalBook}*{\S 4.4}.  If $\M = (E,p)$ is a valuated matroid,
its {\em support} $\{ B \in \binom{E}{r} : p(B) \neq \infty \}$ is the
collection of bases of a rank $r$ matroid on the ground set $E$,
called the {\em underlying matroid} $\underline{\M}$ of $\M$.  The
function $p$ is called the {\em basis valuation function} of $\M$.  We
consider the basis valuation functions $p$ and $\lambda \ttimes p$ with
$\lambda \in R^*$ to be the same valuated matroid.

As for ordinary matroids, valuated matroids have several
different ``cryptomorphic'' definitions, some of which we now recall.
For more information, see \cite{MurotaTamura}. 

Let $\M$ be a valuated
matroid on the ground set $E$ with basis valuation function $p : 
\binom{E}{r} \to \R$.  Given a basis $B$ of $\underline \M$ and an
element $e \in E \setminus B$, the (valuated) \emph{fundamental circuit}
$H(B,e)$ of $\M$ is the element of the $R$-semimodule $\R^E$ 
whose coordinates are given by
\begin{equation} \label{eqtn:fundamentalcircuit}
H(B,e)_{e'} := p(B \cup e - {e'}) / p(B) \qquad \text{for any ${e'} \in E$,}
\end{equation}
where $/$ denotes division in the 
semifield $R$ (subtraction in the case of $\Rbar$), 
$p(B')= \infty$ if $|B'|>r$, and $\infty/\lambda = \infty$ 
for any $\lambda \in \R^*$.  A (valuated) \emph{circuit} of $\M$ is any
vector in $\R^E$ of the form $\lambda \ttimes H(B,e)$, where $B$ is
a basis of $\underline \M$, $e \in E \setminus B$, and $\lambda \in
\R^*$.  We denote by $\C(\M)$ the collection of all circuits of
$\M$.  For any $H \in \R^E$, its {\em support} is defined as
\[
\supp(H) := \{e \in E: H_e \neq \infty\}. 
\]
The set of supports of the circuits of $\M$ is equal to 
the set of circuits of the underlying matroid $\underline{\M}$.
Furthermore, if two circuits $G$ and $H$ of $\M$ have the same support
then there exists $\lambda \in \R^*$ such that $G = \lambda
\ttimes H$ \cite{MurotaTamura}*{Theorem 3.1 (VC$3_e$)}. 

Collections of circuits of valuated matroids can be intrinsically 
characterized by axioms that generalize the classical circuit
axioms for matroids; see \cite{MurotaTamura}*{Theorem 3.1}.
The most important one is the following elimination property.
\begin{itemize}
\item  
{\it Circuit elimination axiom:} 
For any $G, H \in \C(\M)$ and any 
$e,{e'} \in E$ such that 
$G_e = H_e \neq \infty$ and $G_{e'} < H_{e'}$, there exists $F \in \C(\M)$ satisfying
$F_e = \infty$, $F_{e'} = G_{e'}$, and $F \geq G \tplus H$.
\end{itemize}
Here $F \geq F'$ if $F_e \geq F'_e$ for
all $e$, and $(G \tplus H)_e = G_e \tplus H_e$.

We also use the vector formulation for
valuated matroids, which generalizes the notion of cycles for
matroids.  A {\em cycle} of a matroid is a union of circuits. A
\emph{vector} of a valuated matroid is an element of the subsemimodule
of $\R^E$ generated by the valuated circuits. More explicitly, the set
of vectors of $\M$ is
\[
 \V(\M) := \{ {\textstyle \bigoplus_{H \in \C(\M)} \, \lambda_H \ttimes H} : 
\lambda_H \in \R
\text{ for all } H \}.
\]
Vectors of a valuated matroid can also be characterized by axioms;
see \cite{MurotaTamura}*{Theorem 3.4}. In fact, a subset $\V \subseteq R^E$ 
is the set of vectors of a valuated matroid if and only if it 
is a subsemimodule of $R^E$ satisfying the following property.
\begin{itemize}
\item
{\it Vector elimination axiom:} 
For any $G, H \in \V$ and any $e \in E$ such that 
$G_e = H_e \neq \infty$, there exists $F \in \V$ satisfying
$F_e = \infty$, $F \geq G \tplus H$, and $F_{e'} = G_{e'} \tplus H_{e'}$ for all 
${e'} \in E$ such that $G_{e'} \neq H_{e'}$.
\end{itemize}

In the case that $R = \mathbb B$, valuated matroids are simply matroids, as
there is no additional information encoded in the valuation.  In this
case, valuated circuits and vectors are in one-to-one correspondence
with circuits and cycles, respectively,  of the underlying matroid $\underline{\M}$.

In the case that $R = \Rbar$, the set of vectors of a valuated matroid
is also called a {\em tropical linear space} in the tropical
literature. In the terminology used in \cite{TropicalBook}*{\S 4.4},
if $p$ is the basis valuation function of a rank-$r$ valuated matroid
$\M$ then $\V(\M)$ is the tropical linear space $L_{p^\perp}$, where
$p^\perp$ is the dual tropical Pl\"ucker vector given by $p^{\perp}(B) := p(E
\setminus B)$.  

\subsection{Tropical ideals}

We now introduce the main object of study of this paper.  Let $\mon_d$
be the set of monomials of degree $d$ in the variables
$x_0,\dots,x_n$.  We will identify elements of $\R^{\mon_d}$ with
homogeneous polynomials of degree $d$ in $\R[x_0,\dots,x_n]$.  In this
way, if $\M$ is a valuated matroid on the ground set $\mon_d$,
circuits and vectors of $\M$ can be thought of as homogeneous
polynomials in $\R[x_0, \dotsc, x_n]$ of degree $d$.

\begin{definition}\label{d:tropicalidealprojective}
({\em Homogeneous tropical ideals.)}
A \emph{homogeneous tropical ideal} is a homogeneous ideal $I \subseteq \R[x_0,\dots,x_n]$ such that
for each $d \geq 0$ the degree-$d$ part $I_d$ is
the collection of vectors of a valuated matroid $\M_d$ on $\mon_d$.
If $I \subseteq \R[x_0,\dots,x_n]$ is a homogeneous tropical ideal, 
we will denote by $\M_d(I)$ the valuated 
matroid such that $I_d = \V(\M_d(I))$.
\end{definition}

This definition is consistent with Definition \ref{d:tropicalideal} 
in the introduction, in view of the characterization of vectors of a 
valuated matroid in terms of the vector elimination axiom.

Not all homogeneous ideals in $R[x_0,\dots,x_n]$ are tropical ideals. 
As an example, 
consider the ideal $I$ in $\Rbar[x,y]$ generated by $x \tplus y$.
The degree-two part of this ideal is the $\Rbar$-semimodule generated
by $x^2 \tplus xy$, and $xy \tplus y^2$. 
This is not the set of vectors of a valuated
matroid on $\mon_2 = \{x^2,xy,y^2\}$, as the polynomial $x^2 \tplus y^2$ 
would be required to be in $I$ by the vector elimination axiom 
applied to the two generators.

Homogeneous tropical ideals in the sense of
Definition~\ref{d:tropicalidealprojective} will define subschemes of
tropical projective space.  We elaborate on this definition with
some examples.

\begin{example}({\em Realizable tropical ideals.})
Let $\K$ be a field equipped with a valuation $\val: \K \to R$. 
Any polynomial $g \in \K[x_0, \dotsc, x_n]$ gives rise to a ``tropical''
polynomial $\trop(g) \in R[x_0, \dotsc, x_n]$ by interpreting 
all operations tropically 
and replacing coefficients by their valuations:
if $g = \sum c_{\mathbf{u}} \cdot \mathbf{x}^{\mathbf{u}}$, then 
$\trop(g) := \bigoplus \val(c_{\mathbf{u}}) \ttimes \mathbf{x}^\mathbf{u}$.
For any ideal $J \subseteq \Kx$, its tropicalization is the ideal
\[
\trop(J) := \langle \trop(g) : g \in J \rangle \subseteq R[x_0, \dotsc ,x_n].
\]

If $f$ is a polynomial of minimal support in $\trop(J)$, there exists $g
\in J$ and $\lambda \in R^*$ such that $f = \lambda \ttimes \trop(g)$.  
Indeed, if $f = \sum \lambda_i \ttimes 
\mathbf{x}^{\mathbf u_i} \ttimes \trop(g_i)$ with $g_i \in
J$ for all $i$, then $f = \sum \lambda_i \ttimes \trop(\mathbf{x}^{\mathbf{u}_i} g_i)$. Since
there is no cancellation in $R$, we have $\supp(\mathbf{x}^{\mathbf{u}_i} g_i) \subseteq \supp(f)$
for all $i$, and thus $\supp(\mathbf{x}^{\mathbf{u}_i} g_i) = \supp(f)$ for all $i$.
As all the $\mathbf{x}^{\mathbf{u}_i} g_i$ are of minimal support in $J$, 
this implies that all of them are scalar multiples of each other,
and thus $f = \lambda \ttimes \trop(g)$ for some $g \in J$.  

It follows that if $J \subseteq \Kx$ is a homogeneous ideal, for any
$d \geq 0$ the degree-$d$ polynomials in $\trop(J)$ of minimal support
satisfy the circuit elimination axiom of valuated matroids.  The other circuit axioms are immediate, so
$\trop(J)$ is a homogeneous tropical ideal.  A tropical ideal arising
in this way is called {\em realizable} (over the field $\K$).
\end{example}

\begin{remark}
If $J \subseteq \Kx$ is a homogeneous ideal, 
the degree-$d$ part $\trop(J)_d$ of the homogeneous tropical ideal $\trop(J)$ 
is the $R$-semimodule generated by the polynomials $\trop(g)$ 
with $g \in J_d$.
Moreover, if the residue field of $\K$ is infinite then every polynomial 
in $\trop(J)_d$ is a scalar multiple of some $\trop(g)$ with $g \in J_d$, as
$\trop(g) \tplus \trop(h) = \trop(g + \alpha h)$ for a sufficiently
general $\alpha \in K$ with $\val(\alpha) =0$.
If, in addition, the value group $\Gamma := \im \val$ of $K$ equals
all of $R$ then every polynomial in $\trop(J)_d$ is of the form $\trop(g)$ 
for $g \in J_d$.

When the residue field is finite, however, it is possible for the
underlying matroid $\underline{\M_d(\trop(J))}$ to have some cycles that
are not the support of any polynomial in $J$.  For example, consider
$K=\mathbb Z/2\mathbb Z$ equipped with the trivial valuation $\val: K
\to \mathbb B$, and the ideal $J := \langle x + y \rangle
\subseteq K[x,y]$. In this case there is no polynomial in $J_2$
having support $D := \{x^2, xy, y^2\}$, even though the tropical polynomial $x^2 \tplus
x y \tplus y^2$ is in $\trop(J)_2$ and $D$ is a cycle of
$\underline{\M_2(\trop(J))}$.
\end{remark}

A special class of realizable homogeneous tropical ideals consists of the
tropical equivalent of the homogeneous ideal of a point in projective space.

\begin{example}\label{e:idealofpoint}({\em The homogeneous tropical ideal of a point.}) 
Fix $\mathbf{a} = (a_0,a_1,\dotsc,a_n) \in \trop(\mathbb P^n) =
(\Rbar^{n+1} \setminus \{ (\infty,\dots,\infty) \})/\mathbb R
(1,\dots, 1)$.  Let $I_{\mathbf{a}}$ be the ideal generated by all
homogeneous polynomials $f \in \Rbar[x_0,\dots,x_n]$ for which
$\mathbf{a} \in V(f)$, so $f(\mathbf{a}) = \infty$ or the minimum in
$f(\mathbf{a})$ is achieved at least twice.  We claim that
$I_{\mathbf{a}}$ is a tropical ideal.  In addition, if $K$ is a valued
field with $a_i$ in the image of the valuation for each $0 \leq i \leq
n$, then $I_{\mathbf{a}}$ is the tropicalization of {\em any} ideal
$J_{\alpha}$ of a point $\alpha \in \mathbb P^{n}_K$ with
$\val(\alpha_i)=a_i$.

To prove the first claim, we first note that $I_{\mathbf{a}}$ is
generated as an $\Rbar$-semimodule by the set $\mathcal P$ of
polynomials of the form $(\mathbf{a} \cdot \mathbf v) \ttimes \mathbf
x^{\mathbf u} \tplus (\mathbf{a} \cdot \mathbf u) \ttimes \mathbf
x^{\mathbf v}$ with $\deg(\mathbf x^{\mathbf u})=\deg(\mathbf
x^{\mathbf v})$ and $\mathbf u \neq \mathbf v$, where by convention we
take $\infty$ times $0$ equal to $0$ when some of the $a_i$ are equal
to $\infty$.  Indeed, all these binomials (and monomials, in the case
some of the $a_i$ equal $\infty$) are contained in $I_{\mathbf{a}}$.
Suppose that $f = \bigoplus c_{\mathbf u} \ttimes \mathbf x^{\mathbf
  u} \in I_{\mathbf{a}}$ is outside the ideal generated by $\mathcal
P$.  We may assume that $f$ has been chosen to have as few terms as
possible.  We may also assume that $f(\mathbf{a}) < \infty$, as
otherwise all terms of $f$ lie in the ideal generated by $\mathcal P$.
Fix $c_{\mathbf v} \ttimes \mathbf x^{\mathbf v}$ to be a term of $f$
such that $c_{\mathbf v} + \mathbf{a} \cdot \mathbf v > f(\mathbf{a})$
if one exists, or any term of $f$ otherwise.  Take $\mathbf v' \neq
\mathbf v$ with $c_{\mathbf v'} + \mathbf{a} \cdot \mathbf v' =
f(\mathbf{a})$, and set $g = \bigoplus_{\mathbf u \neq \mathbf v}
c_{\mathbf u} \ttimes \mathbf x^{\mathbf u}$.  Then $f = g \tplus
(c_{\mathbf v}- \mathbf{a}\cdot \mathbf v') \ttimes ((\mathbf{a} \cdot
\mathbf v') \ttimes \mathbf x^{\mathbf v} \tplus (\mathbf{a} \cdot
\mathbf v) \ttimes \mathbf x^{\mathbf v'})$.  In the case that the
minimum in $f(\mathbf{a})$ is not achieved at $\mathbf x^{\mathbf v}$,
the minimum in $g(\mathbf{a})$ is still achieved at least twice. If
the minimum in $f(\mathbf a)$ is achieved at all its terms, $f$ must
have at least three terms, as otherwise it would be a scalar multiple
of one of the polynomials in $\mathcal P$. The minimum in $g(\mathbf
a)$ is thus still achieved twice.  Furthermore, $g$ has fewer terms
than $f$, so by assumption $g$ lies in the $\Rbar$-semimodule
generated by $\mathcal P$, and thus so does $f$.  This proves that
$I_{\mathbf{a}}$ is generated by $\mathcal P$.  In addition, the
degree-$d$ polynomials of minimal support in $\mathcal P$ satisfy the
valuated circuit elimination axiom, which implies that
$(I_{\mathbf{a}})_d$ is the set of vectors of a valuated matroid
\cite{MurotaTamura}*{Theorem~3.4}, and thus $I_{\mathbf{a}}$ is a
tropical ideal.

Suppose now that $\alpha \in \mathbb P^{n}_K$ satisfies $\val(\alpha) =
\mathbf{a}$.  The homogeneous ideal $J_{\alpha} \subseteq K[x_0, \dotsc,
  x_n]$ of the point $\alpha$ contains the binomials $\alpha^{\mathbf
  v} \mathbf x^{\mathbf u} - \alpha^{\mathbf u} \mathbf x^{\mathbf v}$
for all pairs $\mathbf x^{\mathbf u}, \mathbf x^{\mathbf v}$ with
$\deg(\mathbf x^{\mathbf u}) = \deg(\mathbf x^{\mathbf v})$, so
$I_{\mathbf{a}} \subseteq \trop(J_{\alpha})$.  If the inclusion were
proper, there would be $h \in \trop(J_{\alpha})$ with $h \notin
I_{\mathbf a}$.  Write $h = \sum h_i \ttimes \trop(g_i)$, where $g_i \in
J_{\alpha}$.  Since $\mathbf{a} = \val(\alpha) \in \trop(V(g_i)) =
V(\trop(g_i))$, this contradicts that $\mathbf{a} \not \in V(h)$.
\end{example}

One can think of a homogeneous tropical ideal as the ``tower'' 
of valuated matroids that determine its various homogeneous parts, 
as described in the following definition.

\begin{definition}({\em Compatible valuated matroids.})
Let $\mathcal S = (\M_d)_{d \geq 0}$ be a sequence of valuated
matroids with values in the semiring $R$, where the ground set of
$\M_d$ is the set of monomials $\mon_d$.  The sequence $\mathcal S$ is
called a {\em compatible} sequence if the $\R$-subsemimodule $I$ of $R[x_0,\dots,x_n]$
generated by $\{ \V(\M_d) : d \geq 0 \}$
is a homogeneous ideal.
In other words, $\mathcal S$ is compatible if
for any $d \geq 0$ and any variable $x_i$ we have
\[ x_i \ttimes \V(\M_d) := \{ x_i \ttimes H : H \in \V(\M_d) \} \subseteq 
\V(\M_{d+1}),\]
or equivalently, for any circuit $G \in \C(\M_d)$ we have that $x_i \ttimes G$
is a sum of circuits of $\M_{d+1}$.
\end{definition}

Compatibility of valuated matroids can be simply described in terms 
of their basis valuation functions.

\begin{proposition}({Compatibility in terms of basis valuations.})
Let $\mathcal S = (\M_d)_{d \geq 0}$ be a sequence of valuated matroids,
where $\M_d = (\mon_d, p_d)$ has rank $r_d$.
The sequence $\mathcal S$ is a compatible sequence if and only if for any $d \geq 0$,
any variable $x_i$, any $U \subseteq \mon_d$ 
of size $r_d+1$, and any $V \subseteq \mon_{d+1}$ of size $r_{d+1}-1$, we have 
that 
\[
\min_{\mathbf{x}^\mathbf{u} \in x_iU \setminus V} \,\, p_d(U - 
\mathbf{x}^\mathbf{u}/x_i) \ttimes 
p_{d+1}(V \cup \mathbf{x}^\mathbf{u}) \,\,\text{ is attained at least twice.}
\]
\end{proposition}
\begin{proof}
For any $d \geq 0$ and any variable $x_i$, the set $x_i \ttimes \V(\M_d)$ is the 
collection of vectors of a valuated matroid on the ground set $\mon_{d+1}$. 
Indeed, if we let $\mon_d^{\hat i}$ be the set 
of monomials in $\mon_{d+1}$ not divisible by $x_i$, the collection 
$x_i \ttimes \V(\M_d)$ is the set of vectors of the valuated matroid $\M_{d,i}$
on $\mon_{d+1}$ with basis valuation function
\[
 p_{d,i} (B) = 
 \begin{cases}
  p_d((B \setminus \mon_d^{\hat i})/x_i) &\text{if } B \supseteq \mon_d^{\hat i}, \\
  \infty &\text{ otherwise.}
 \end{cases} 
\]
The sequence $\mathcal S$ is compatible 
if and only if for every $d \geq 0$ and any variable $x_i$ 
there is a containment of tropical linear
spaces $\V(\M_{d,i}) \subseteq \V(\M_{d+1})$, or dually, $\V(\M_{d,i})^{\perp} 
\supseteq \V(\M_{d+1})^{\perp}$. By \cite{Haque}*{Theorem~1}, this is equivalent
to the condition that for any $V, V' \subseteq \mon_{d+1}$ 
with $|V| = |\mon_d^{\hat i}|+r_d+1$ and $|V'| = r_{d+1}-1$,
\[
 \min_{\mathbf{v} \in V\setminus V'} \,\, p_{d,i}(V - \mathbf{v}) \ttimes 
 p_{d+1}(V' \cup \mathbf{v}) \,\, \text{ is attained at least twice,}
\]
which reduces to the desired condition.
\end{proof}

The next example shows that tropical ideals carry strictly more information
than their varieties.
\begin{example}({\em Two tropical ideals with the same variety.})
\label{e:samevariety}
Let $K$ be the field $\mathbb C$ equipped with the trivial valuation
$\val: \mathbb C \to \Rbar$.  Consider the principal ideals $J =
\langle (x + y + z)(xy+ xz+yz) \rangle$ and $J' = \langle
(x+y)(x+z)(y+z) \rangle$ in $\mathbb C[x,y,z]$. Their tropicalizations
$I := \trop(J)$ and $I' := \trop(J')$ are homogeneous tropical ideals
in $\Rbar[x,y,z]$ with the same variety. This is shown in Figure
\ref{f:irreddecomp}, together with the two associated tropical
irreducible decompositions corresponding to tropicalizing the factors
in the generators of $J$ and $J'$.
\begin{figure}[htb]
 \begin{center}
  \includegraphics[scale=0.55]{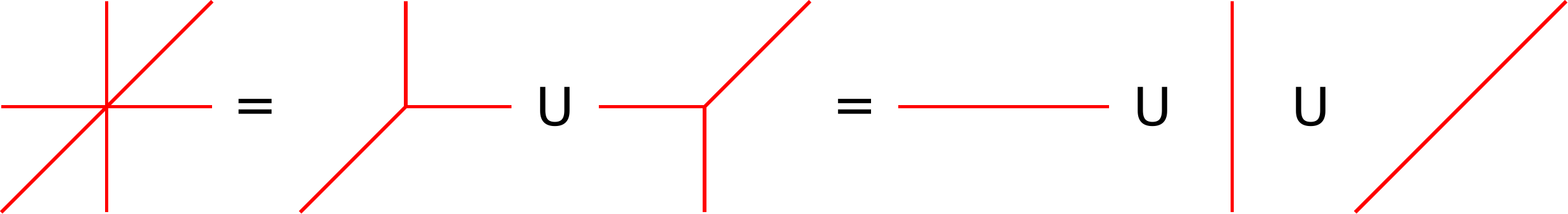}
	 \caption{   \label{f:irreddecomp}
Two tropical irreducible decompositions.}
 \end{center}
\end{figure}
However, even though the tropical ideals $I$ and $I'$
coincide up to degree 3, they are distinct.  For instance, the degree
4 polynomial
\[
 f := x^3 y \tplus x^3 z \tplus x y^3 \tplus y^3  z 
 \tplus x  z^3 \tplus y  z^3 \tplus x  y^2  z \tplus
 x  y  z^2 \tplus y^2  z^2
\]
is equal to $\trop((x+y)(x+z)(y+z)(x-y-z)) \in I'$, but $f$ is not 
in $I$: a simple computation shows that no polynomial of the form 
$(x + y + z)(xy+ xz+yz)(ax+by+cz)$ can have support equal to the support of $f$. 
\end{example}

We finish this section with an example of a tropical ideal that is not 
realizable over any field $\K$.
\begin{example}\label{ex:non-realizable}
({\em A non-realizable tropical ideal.})
For any $n \geq 2$, we give an example of a homogeneous tropical ideal in 
$\Rbar[x_0, \dotsc, x_n]$ that is not realizable over any valued field.
Its variety is, however, 
the standard tropical line in tropical projective space;
see Example \ref{e:varietyofnonrealizable}.

For $d \geq 0$, let $\M_d$ be the rank-($d+1$) valuated matroid 
on the ground set $\mon_d$ whose basis valuation function 
$p_d : \binom{\mon_d}{d+1} \to \Rbar$ is given by
\[
p_d(B) := 
\begin{cases}
0 & \text{if for any } k \leq d \text{ and any } 
\mathbf{x}^\mathbf{v} \in \mon_k \text{ we have } \\
& \qquad \left|\{\mathbf{x}^\mathbf{u} \in B :
\mathbf{x}^\mathbf{v} \text{ divides } \mathbf{x}^\mathbf{u} \}\right| 
\leq d - k + 1,\\
\infty & \text{otherwise}.
\end{cases}
\]
Geometrically, if we think of $\mon_d$ as the set of lattice points 
inside a simplex of size $d$, 
the function $p_d$ assigns the value $0$ precisely to those $(d+1)$-subsets $B$
of $\mon_d$ such that for any $k \leq d$, the subset $B$ contains 
at most $d-k+1$ monomials from any simplex in $\mon_d$ of size $d-k$.
A proof that the underlying matroids $\underline{\M_d}$ are indeed matroids 
(and thus the $\M_d$ are valuated matroids) can be found in 
\cite{AB}*{Theorem~4.1}, where these matroids are studied 
in connection to flag arrangements.

The circuits of $\M_d$ are the tropical polynomials of the form $H =
\lambda \ttimes \bigoplus_{\mathbf u \in C} \mathbf{x}^\mathbf{u}$
with $\lambda \in \mathbb R$ and $C$ an inclusion-minimal subset of
$\mon_d$ satisfying $|C| > d - \deg(\gcd(C)) + 1$.  This description
shows that if $H$ is a circuit of $\M_d$ and $x_i$ is any variable
then $x_i \ttimes H$ is a circuit of $\M_{d+1}$.  Thus $(\M_d)_{d\geq
  0}$ is a compatible sequence of valuated matroids, and so the
$\Rbar$-semimodule $I$ generated by all $\V(\M_d)$ 
is a homogeneous tropical ideal.

To show that $I$ is a non-realizable tropical ideal, note that the
tropical polynomial $f = x_0 \tplus x_1 \tplus x_2$ is a circuit of
$\M_1$, so in particular $f \in I$.  If $I$ is realizable then $I =
\trop(J)$ for some ideal $J \subseteq \K[x_0, \dots, x_n]$, and so,
since $f$ is a circuit, there exists some $g \in J$ such that $f =
\trop(g)$.  After a suitable scaling of the variables, we may assume
that $g = x_0 + x_1 + x_2$. The polynomial
\[
h := g \cdot (x_0^2 + x_1^2 + x_2^2 - x_0x_1 - x_0x_2 - x_1x_2) 
= x_0^3 + x_1^3 + x_2^3 - 3x_0x_1x_2
\]
is then a polynomial in $J$, and thus $\trop(h)$ lies in $I$. 
However, this contradicts the fact that
$\supp(\trop(h)) \subseteq \{x_0^3, x_1^3, x_2^3, 
x_0 x_1 x_2\}$ is an 
independent set in the underlying matroid $\underline{\M_3}$.
\end{example}

\section{Gr\"obner theory for tropical ideals}\label{sec:grobner}

\label{s:3}

In this section we develop a Gr\"obner theory for homogeneous tropical
ideals in $\Rbar[x_0,\dots,x_n]$ and $\mathbb B[x_0,\dots,x_n]$, 
and use it to prove some
basic properties of tropical ideals with coefficients in a more
general semiring $R$. These include the eventual polynomiality of
their Hilbert functions, and the fact that tropical ideals satisfy the
ascending chain condition.

We start by defining initial ideals of homogeneous ideals in 
$\Rbar[x_0, \dotsc, x_n]$
with respect to a weight vector $\mathbf{w} \in \Rbar^{n+1}$.

\begin{definition} \label{d:initialIdeal}({\em Initial ideals.})
Let $\mathbf w \in \Rbar^{n+1}$ with $\mathbf{w} \neq (\infty, \dots, \infty)$.
The \emph{initial term} of a tropical polynomial $f = \bigoplus a_{\mathbf{u}}
\ttimes \mathbf{x}^{\mathbf{u}} \in \Rbar[x_0,\dots,x_n]$ with respect
to $\mathbf{w}$ is the tropical polynomial in $\mathbb B[x_0,\dots,x_n]$
given by
\[\inn_{\mathbf{w}}(f) := 
\begin{cases}
\bigoplus_{a_{\mathbf{u}} + \mathbf{w} \cdot \mathbf{u} = 
f(\mathbf{w})} \mathbf{x}^{\mathbf{u}} & \text{if } f(\mathbf{w}) < \infty, \\
\infty & \text{otherwise.}
\end{cases}
\quad \in \mathbb B[x_0,\dots,x_n]
\] 
In computing the dot product  $\mathbf{w}
\cdot \mathbf{u}$ we follow the convention that $\infty$ times $a$ is 
$\infty$ for $a \neq 0$, and $\infty$ times  $0$ is $0$, so
$\mathbf{w} \cdot \mathbf{u} = \mathbf{w}^{\mathbf{u}}$.  
The \emph{initial ideal} $\inn_{\mathbf w}(I)$ 
of a homogeneous ideal $I$ in $\Rbar[x_0,\dots,x_n]$ is the homogeneous ideal
\[\inn_{\mathbf{w}}(I) := \langle \inn_{\mathbf{w}}(f) : f \in I
\rangle \subseteq \mathbb B[x_0,\dots,x_n].\]
If $I$ is a homogeneous ideal in $\mathbb B[x_0,\dots,x_n]$ then we
define its initial ideal $\inn_{\mathbf{w}}(I)$ as the initial ideal
of $I\Rbar[x_0,\dots,x_n]$, using the inclusion of $\mathbb B$ into $\Rbar$. 
\end{definition}

Note that $\inn_{\mathbf{w}}( a \ttimes f \tplus b \ttimes g) = 
\inn_{\mathbf{w}}(f)
\tplus \inn_{\mathbf{w}}(g)$ for $f, g \in I$ and $a,b \in \Rbar$ such that $a
\ttimes f(\mathbf{w}) = b \ttimes g(\mathbf{w}) < \infty$, so the
set of initial forms of a graded piece $I_d$ of $I$ is a semimodule
over $\mathbb B$.  As in addition $\inn_{\mathbf{w}}(x_i\ttimes f) = x_i \ttimes
\inn_{\mathbf{w}}(f)$, 
we have $\inn_{\mathbf{w}}(I_d) = (\inn_{\mathbf{w}}I)_d$.
Also, since $I$ is a homogeneous ideal, $\inn_{\mathbf{w}}(I) =
\inn_{\mathbf{w}+\lambda \mathbf{1}}(I)$ for $\mathbf{1} =
(1,1,\dots,1)$, so we may regard $\mathbf{w}$ as an element of
$\trop(\mathbb P^n)$.

\begin{remark}
Our definition of initial ideals is compatible with the usual
definition of initial ideals used in tropical geometry, in the sense
that for any homogeneous ideal $J \subseteq \K[x_0,\dots,x_n]$ we have
$\inn_\mathbf{w} (\trop(J)) = \trop(\inn_\mathbf{w} (J))$. The initial ideal
$\inn_{\mathbf{w}}(J)$ is an ideal in the polynomial ring with
coefficients in the residue field of $\K$, which has a trivial
valuation, so $\trop(\inn_{\mathbf{w}}(J))$ is naturally an ideal in
$\mathbb B[x_0,\dots,x_n]$.
\end{remark}

An important result in this section is that initial ideals of tropical
ideals are also tropical ideals. This will follow from the following
key fact about valuated matroids.

We first extend the notion of initial term to vectors of any valuated matroid 
over $\Rbar$. If $E$ is any finite set, 
$H \in \Rbar^E$, and $\mathbf w \in \Rbar^E$ with 
$\mathbf{w} \neq (\infty, \dotsc, \infty)$, 
the {\em initial term} of $H$ with respect to $\mathbf w$ is the subset of $E$
\[
\inn_{\mathbf w} H := \{ e \in E : H_e + w_e 
\text{ is minimal among all } e \in E \}
\]
if $\min(H_e + w_e) < \infty$, and $\inn_{\mathbf w} H := \emptyset$ otherwise. 

\begin{lemma}\label{lem:initial}
Let $\M = (E,p)$ be a rank-$r$ valuated matroid, 
where $p: \binom{E}{r} \to \Rbar$ is its basis valuation function,
and let $\mathbf w = (w_e)_{e \in E} \in \mathbb R^E$. Then 
\[
\textstyle \inn_{\mathbf w} \B(\M) := \left\{ B \in \binom{E}{r} : p(B) - 
\sum_{e \in B} w_e \text{ is minimal among all } B \in \binom{E}{r} \right\}
\]
is the collection of bases of an (ordinary) matroid $\inn_{\mathbf w} \M$ 
of rank $r$ on the ground set $E$.
Its circuits are the elements of 
$\inn_{\mathbf w}\C(\M) := \{ \inn_{\mathbf w} H  : H \in \C(\M) 
\}$ that are minimal with respect to inclusion, 
and its set of cycles is
\[
\inn_{\mathbf w}\V(\M) := \{ \inn_{\mathbf w} H  : H \in \V(\M) \}.
\]
\end{lemma}
\begin{proof}
For any $B \in \binom{E}{r}$ set 
\[
\textstyle p_{\mathbf w}(B) :=  p(B) - \sum_{e \in B} w_e.
\]
To prove that $\inn_{\mathbf w} \B(\M)$ satisfies the basis exchange
axiom, fix $A, B \in \binom{E}{r}$ and choose $a \in A
\setminus B$.  Since $p$ satisfies the valuated basis exchange axiom
\eqref{eqtn:tropicalplucker} there exists $b \in B\setminus A$ such
that
\[p(A) + p(B) \geq p(A \cup b - a) + p(B \cup a - b).\]
Subtracting $\sum_{e \in A}w_e + \sum_{e \in B}w_e$ on both sides we get
\begin{equation*} \label{eqtn:initialweight}
 p_{\mathbf{w}}(A) + p_{\mathbf{w}}(B) 
\geq p_{\mathbf{w}}(A \cup b -a) + p_{\mathbf{w}}(B \cup a -b).
\end{equation*}
This implies that $p_{\mathbf w}$ also satisfies the valuated basis
exchange axiom.  Moreover, if $A, B \in \inn_{\mathbf{w}} \B(\M)$ then both $A\cup b-a$ and $B \cup a-b$ are 
in $\inn_{\mathbf w} \B(\M)$ 
as well.

We now prove that the circuits of $\inn_{\mathbf w} \M$ have the
desired description.  The $\mathbf w$-initial term of any fundamental
circuit $H(B,e)$ of $\M$ (as in \eqref{eqtn:fundamentalcircuit}) is
the set of $e' \in E$ for which $p(B \cup e - e') - p(B) + w_{e'}$ is
minimal.  Adding $p(B) - \sum_{a \in B \cup e} w_a$, we get
\begin{equation}\label{eq:initial}
\inn_{\mathbf w} H(B,e) = \{e' \in E : p_{\mathbf w}(B \cup e - e') 
\text{ is minimal among all } e' \in E \}. \\
\end{equation}
Any circuit $C$ of the initial matroid $\inn_{\mathbf w}\M$ 
is the fundamental circuit of an element $e \in E$ over a basis 
$B \in \inn_{\mathbf w} \B(\M)$, and thus has the form
\begin{align*}
C &= \{e' \in E : B \cup e - e' \in \inn_{\mathbf w} \B(\M) \} \\
&= \{e' \in E : p_{\mathbf w}(B \cup e - e') = p_{\mathbf w}(B) \} \\
&= \inn_{\mathbf w} H(B,e).
\end{align*}
This shows that all circuits of $\inn_{\mathbf w} \M$ appear in the
collection $\inn_{\mathbf w}\C(\M)$.  To prove our claim, it then
suffices to show that each set in $\inn_{\mathbf w}\C(\M)$ is a
dependent set of $\inn_{\mathbf w} \M$.  Assume by contradiction that
$\inn_{\mathbf w} H \subseteq B$ for some $H \in \C(\M)$ and some
basis $B \in \inn_{\mathbf w} \B(\M)$.  Fix $e \in \inn_{\mathbf w}
H$, and a basis $B'$ of the underlying matroid $\underline{\M}$
such that $H = \lambda \ttimes H(B',e)$ for some $\lambda \in \mathbb
R$.  Since $e \in B \setminus B'$, the valuated basis exchange axiom for $p_{\mathbf{w}}$ implies that 
there exists $e' \in B' \setminus B$ such that
\begin{equation}\label{eq:exchange}
p_{\mathbf{w}}(B) + p_{\mathbf{w}}(B') 
\geq p_{\mathbf{w}}(B \cup e' -e) + p_{\mathbf{w}}(B' \cup e -e').
\end{equation}
As $\inn_{\mathbf{w}} H(B',e) \subseteq B$ and $e' \not \in B$, by \eqref{eq:initial}
we have $p_{\mathbf{w}}(B') < p_{\mathbf{w}}(B' \cup e -e')$.
Moreover, since $B \in \inn_{\mathbf w} \B(\M)$, we also have
$p_{\mathbf{w}}(B) \leq p_{\mathbf{w}}(B \cup e' -e)$.  But then,
adding these two inequalities gives a contradiction to
\eqref{eq:exchange}.

We now prove that the set of cycles of $\inn_{\mathbf w} \M$ 
is equal to $\inn_{\mathbf w}\V(\M)$. 
Any $H \in \V(\M)$ has the form $H = H^1 \tplus \dotsb \tplus H^k$ 
for circuits $H^1, \dotsc, H^k$ of $\M$.
Its initial form $\inn_{\mathbf w} H$ is the union of those initial 
forms $\inn_{\mathbf w} H^i$ for which
$\min_{e \in E}(H^i_{e} + w_{e}) = \min_{e \in E}(H_{e} + w_{e})$. 
It follows that the collection $\inn_{\mathbf w}\V(\M)$ is closed under unions. 
Moreover, our previous claim implies that $\inn_{\mathbf w}\V(\M)$ 
contains all circuits of $\inn_{\mathbf w}\M$. To complete the proof, 
it is thus enough to show that for any circuit $G$ of $\M$, 
the set $\inn_{\mathbf w} G$ is a cycle of $\inn_{\mathbf w}\M$. 

We proceed by contradiction. Assume $G$ is a circuit of $\M$ such that
$\inn_{\mathbf w} G$ is not a cycle, and take $G$ such that
$\inn_{\mathbf w} G$ is inclusion-minimal with this property.  Since
$\inn_{\mathbf w} G$ is dependent in $\inn_{\mathbf w}\M$, it contains
a circuit $C$ of $\inn_{\mathbf w}\M$, which must have the form $C =
\inn_{\mathbf w} H$ for some circuit $H$ of $\M$.  After tropically
rescaling, we may assume that $\min_{e \in E}(G_e +w_e) = \min_{e \in
  E}(H_e +w_e)$.  Let $e \in \inn_{\mathbf w} H \subseteq
\inn_{\mathbf w} G$.  We have $G_e + w_e = H_e + w_e$, and so $G_e =
H_e$.  Let $e' \in \inn_{\mathbf w} G$ be such that $e'$ is in no
circuit of $\inn_{\mathbf w} \M$ contained in $\inn_{\mathbf w} G$; in
particular, $e' \notin \inn_{\mathbf w} H$.  It follows that $G_{e'} + w_{e'}
< H_{e'} + w_{e'}$, and so $G_{e'} < H_{e'}$.  We can then apply the circuit
elimination axiom to get a circuit $F$ of $\M$ such that $F_e =
\infty$, $F_{e'} = G_{e'}$, and $F \geq G \tplus H$.  This implies that $e'
\in \inn_{\mathbf w} F$, $\inn_{\mathbf w} F \subseteq \inn_{\mathbf
  w} G$, and $e \notin \inn_{\mathbf w}F$.  We thus have that
$\inn_{\mathbf w} F$ is a proper subset of $\inn_{\mathbf w} G$, and
our minimality assumption implies that $\inn_{\mathbf w} F$ is a
cycle. This contradicts our choice of $e'$, as $e' \in \inn_{\mathbf
  w} F$.
\end{proof}

In order to use Lemma \ref{lem:initial} to describe the relationship between
the valuated matroids of an ideal and those of its initial ideals, we need the
notion of contraction of a valuated matroid.

Let $\mathcal M$ be a valuated matroid on the ground set $E$, 
and let $A$ be a subset of $E$ of rank $s$.  Fix a basis $B_{A}$ of
the restriction $\underline{\mathcal M}|{A}$.  The \emph{contraction} $\mathcal M/A$ 
with respect to $B_{A}$ is the valuated matroid $\mathcal M/A = (E \setminus A,
q)$ of rank $r-s$ with basis valuation function $q : \binom{E \setminus A}{r-s}
\rightarrow R$ given by
\[q(B) = p(B \sqcup B_{A}).\] If we replace $B_{A}$ by a different
basis $B'_{A}$ of $A$ then $q$ changes by global (tropical) multiplication by 
a nonzero constant; see, for instance, \cite{MurotaBook}*{Theorem 5.2.5}. 
The resulting valuated matroids are thus the same,
and we suppress the choice of $B_{A}$ from the notation for contractions. 
The set of vectors of the contraction $\mathcal M/A$ is 
\[\V(\M/A) = \{ H[E \setminus A] : H \in \V(\M) \},\]
where $H[E \setminus A] \in \R^{E \setminus A}$ is given by
$(H[E \setminus A])_e := H_e$ for any $e \in E\setminus A$; 
see, for example, \cite{Baker}*{Theorem 4.17 (2)}.

If $I$ is a homogeneous tropical ideal in $\mathbb B[x_0, \dotsc,
  x_n]$, we write $M_d(I) := \underline{\M_d(I)}$ to emphasize the
fact that no extra information is encoded in the valuated matroid
$\M_d(I)$ besides its underlying matroid.  Also, if $M$ is a matroid
on the ground set $E$, and $F$ is a finite set, the {\em coloop extension} $M
\oplus F$ is the matroid on the ground set $E \sqcup F$ obtained by
attaching to $M$ all the elements of $F$ as coloops, i.e., $\B(M
\oplus F) = \{B \sqcup F: B \in \B(M)\}$.

\begin{theorem}\label{t:initial} 
(Initial ideals of tropical ideals are tropical.)
Let $I$ be a homogeneous tropical ideal in $\Rbar[x_0,\dots,x_n]$, and
fix $\mathbf w \in \Rbar^{n+1}$ with $\mathbf w \neq (\infty, \dotsc, \infty)$.  The
initial ideal $\inn_{\mathbf w}(I)$ is a homogeneous tropical ideal in 
$\mathbb B[x_0,\dots,x_n]$. If $\sigma := \{ i : w_i = \infty\}$ 
and $\mon^\sigma_d$ is the set of monomials in $\mon_d$ divisible by a 
variable $x_i$ with $i \in \sigma$, 
the associated matroid $M_d(\inn_{\mathbf w}(I))$ is equal to 
$\inn_{\hat{\mathbf w}}(\M_d(I)/\mon^\sigma_d) \oplus \mon^\sigma_d$, 
where $\hat{\mathbf w} \in \Rbar^{\mon_d \setminus \mon^\sigma_d}$ is given by 
$\hat{\mathbf{w}}_{\mathbf{x}^{\mathbf{u}}} := \mathbf w \cdot \mathbf u$.
\end{theorem}
\begin{proof}
For any $f \in I_d$, denote by $f^{\sigma}$ the tropical polynomial
obtained by deleting all terms of $f$ involving a monomial in $\mon^\sigma_d$.
By definition, we have $\inn_{\mathbf{w}}(f) = 
\inn_{\mathbf{w}}(f^{\sigma})$.
Moreover, the set $I^{\sigma}_d= \{ f^{\sigma} : f \in I_d \}$ is the set of 
vectors of the valuated matroid $\M_d(I)/\mon^\sigma_d$. 
It follows from Lemma~\ref{lem:initial} that $(\inn_{\mathbf{w}}(I))_d = 
\{ \inn_{\mathbf{w}}(f^{\sigma}) : f \in I_d \}$ is the set 
of cycles of the matroid 
$\inn_{\hat{\mathbf{w}}}(\M_d(I)/\mon^\sigma_d)$,
which is also the set of cycles of the coloop extension
$\inn_{\hat{\mathbf{w}}}(\M_d(I)/\mon^\sigma_d) \oplus \mon^\sigma_d$.
\end{proof}

We now apply these results
to study Hilbert functions of homogeneous tropical ideals in
$R[x_0,\dots,x_n]$, where $R$ is again a general valuative semifield,
and to prove that they satisfy the ascending chain condition.

\begin{definition}({\em Hilbert functions.})
Let $I$ be a homogeneous tropical ideal in $\R[x_0,\dots,x_n]$. 
Its {\em Hilbert function} is 
the function $H_I : \mathbb{Z}_{\geq 0} \to \mathbb{Z}_{\geq 0}$ defined by
\[
H_I(d) := \text{rank}( \M_d(I)).
\]
\end{definition}

If $J \subseteq \Kx$ is a homogeneous ideal then the matroid
$\underline{\M_d(\trop(J))}$ 
encodes the dependencies in $(K[x_0,
  \dotsc, x_n]/J)_d$ among the monomials in $\mon_d$;
a subset of $\mon_d$ is dependent if and only if there is a
polynomial $f \in J_d$ whose support is contained in this subset. 
This implies that the bases of 
$\underline{\M_d(\trop(J))}$ correspond to monomial bases of the vector space $(\Kx/J)_d$.  
It follows that Hilbert functions 
are preserved under tropicalization: 
\[H_{\trop(J)}(d) = \dim_{\K} ((\Kx/J)_d)\]
for any $d \in \mathbb{Z}_{\geq 0}$. See also \cite{Giansiracusa2}*{Theorem 7.1.6}.

A key fact about standard Gr\"obner theory is that the Hilbert function
of an ideal is preserved when passing to an initial ideal.  
We next observe that this also holds in the tropical setting.  
Note that $\mathbf{w}$ is
not allowed to have infinite coordinates for this result.

\begin{corollary}\label{cor:samehilbert}
Suppose $I$ is a homogeneous tropical ideal
in $\Sx$ or in $\Bx$, and fix $\mathbf w \in \mathbb R^{n+1}$. 
For any $d \in \mathbb{Z}_{\geq 0}$ we have 
$H_I(d) = H_{\inn_{\mathbf w} (I)}(d)$.
\end{corollary}
\begin{proof}
Suppose $I$ is a homogeneous tropical ideal in $\Sx$. By Theorem \ref{t:initial}, 
the matroid $M_d(\inn_{\mathbf w}(I))$ is equal to $\inn_{\hat{\mathbf 
w}}\M_d(I)$.
The description given in Lemma \ref{lem:initial}
shows that all bases of $\inn_{\hat{\mathbf w}}\M_d(I)$ 
are bases of the underlying matroid $\underline{\M_d(I)}$, and thus the
rank of $\M_d(I)$ is equal to the rank of $M_d(\inn_{\mathbf w}(I))$, as 
desired.

If $I$ is a homogeneous tropical ideal in $\Bx$, let 
$I' := I\Rbar[x_0,\dots,x_n]$ be the ideal generated by $I$ 
in $\Sx$ using the inclusion of $\mathbb B$ into
$\Rbar$. By definition, the ideal $I'$ is a homogeneous tropical ideal satisfying 
$\inn_{\mathbf w}(I) = \inn_{\mathbf w} (I')$.
Moreover, $I'$ has the same Hilbert function as $I$, as $\underline{\M_d(I')} = 
M_d(I)$. The result then follows from the previous case.
\end{proof}

The following result will be useful in our study of tropical ideals 
and their Hilbert functions.

\begin{lemma}\label{l:monomial}
If $I$ is a homogeneous tropical ideal in $\Sx$ or $\Bx$ then there exists
$\mathbf w \in \mathbb R^{n+1}$ such that $\inn_{\mathbf w}(I)$ is
generated by monomials. In fact, the set of $\mathbf w$ for which 
$\inn_{\mathbf w}(I)$ is not generated by monomials is contained 
in a countable union of hyperplanes in $\mathbb R^{n+1}$.
\end{lemma}

\begin{proof}
Lemma \ref{lem:initial} implies that $\inn_{\mathbf w}(I)$ 
is generated by the polynomials $\inn_{\mathbf w}( f)$
with $f$ a circuit of one of the matroids $\M_d(I)$. 
It thus suffices to show that there is a $\mathbf w \in \mathbb R^{n+1}$
such that for any $d \geq 0$ and any circuit $f$ of $\M_d(I)$, 
the initial term $\inn_{\mathbf w} (f)$ is a monomial.
For any such $f$, the set of $\mathbf w \in \mathbb R^{n+1}$ 
for which $\inn_{\mathbf w}(f)$ is not a monomial is a finite union
of codimension-one polyhedra in $\mathbb R^{n+1}$. 
Moreover, for any $d \geq 0$ the set of circuits of $\M_d(I)$ 
is finite (up to scaling).
It follows that the set of $\mathbf w \in \mathbb R^{n+1}$ 
for which there is some circuit $f$ such that $\inn_{\mathbf w}(f)$ 
is not a monomial is a countable union of codimension-one polyhedra in 
$\mathbb R^{n+1}$.
\end{proof}

In Section \ref{s:5} we will see that, in fact, the set of $\mathbf w \in \mathbb R^{n+1}$
for which $\inn_{\mathbf w}(I)$ is not generated by monomials is a finite union
of polyhedra of codimension at least 1.

The following result shows that Hilbert functions of tropical ideals
are eventually polynomial.

\begin{proposition}\label{prop:eventuallypoly}
If $I$ is a homogeneous tropical ideal in $R[x_0,\dots,x_n]$
then its Hilbert function $H_I$ is eventually polynomial.
\end{proposition}
\begin{proof}
Let $\varphi : R \to \mathbb B$ be the function defined by $\varphi(a)
= 0$ if $a \neq \infty$, and $\varphi(\infty) = \infty$.  Note that $\varphi$ is
a surjective semiring homomorphism, as the assumption on $R$ that
$a \tplus b = \min(a,b)$ implies that $a \tplus b \neq \infty$ if $a,b
\neq \infty$.  This induces a map of polynomial semirings $\varphi :
R[x_0, \dotsc, x_n] \to \mathbb B[x_0, \dotsc, x_n]$.  The image
$\varphi(I)$ of $I$ is a homogeneous tropical ideal in $\mathbb B
[x_0, \dotsc, x_n]$; in fact, we have $M_d(\varphi(I)) =
\underline{\M_d(I)}$. The homogeneous tropical ideals $I$ and $\varphi(I)$
therefore have the same Hilbert function, and thus it is enough to
prove that Hilbert functions of homogeneous tropical ideals in $\mathbb B[x_0,
  \dotsc, x_n]$ are eventually polynomial.

In this case, we can use Lemma \ref{l:monomial} and Corollary
\ref{cor:samehilbert} to reduce to the case where $I$ is a tropical
ideal in $\Bx$ generated by monomials.  In this situation, the Hilbert
function $H_I(d)$ equals the number of monomials of degree $d$ not in
$I$, which in turn equals the Hilbert function of the ideal $J
\subseteq \K[x_0,\dots,x_n]$ generated by the monomials in $I$, where
$\K$ is an arbitrary field.  The result then follows from the standard
fact that the Hilbert function of a homogeneous ideal in a polynomial
ring with coefficients in a field is eventually polynomial.
\end{proof}

\begin{definition} \label{d:HilbpolyDimension}({\em Hilbert polynomials.})
Let $I$ be a homogeneous tropical ideal.
The {\em Hilbert polynomial} of $I$ is the polynomial $P_I$ 
that agrees with the Hilbert function $H_I$ for $d \gg 0$. 
The {\em dimension} of $I$ is the degree of $P_I$.
\end{definition}

The following example shows that tropical ideals are typically not
finitely generated.  However, we prove in
Theorem~\ref{p:ascendingchain} that they do have some Noetherian
properties, in the sense that they satisfy the ascending chain
condition.  

\begin{example}\label{e:notFinitelyGenerated}
({\em Tropical ideals need not be finitely generated.})  Let $J := \langle
x-y \rangle \subseteq K[x,y]$, and let $I := \trop(J) \subseteq
R[x,y]$. We claim that the homogeneous tropical ideal $I$ is not finitely generated.
Note that $x^d-y^d \in J$ for all $d \geq 1$, so $x^d
\tplus y^d \in I$ for all $d \geq 1$.  Suppose that $f_1,\dots,f_r$
is a finite homogeneous generating set for $I$.  Then for all $d
\geq 1$ we can write
\begin{equation} \label{eq:xdplusyd} 
x^d \tplus y^d = \bigoplus h_{id} \ttimes f_i.
\end{equation}
Since there is no cancellation in $R[x,y]$, if $m$ is a monomial
occurring in $h_{id}$ for some $d$, and $m'$ is a monomial occurring in
$f_i$, then $mm'$ occurs in $h_{id}\ttimes f_i$, and thus in $x^d
\tplus y^d$.  This is only possible if each $f_i$ is either $x^d
\tplus y^d$ or a power of $x$ or $y$. 
Thus for any $d \geq \max(\deg(f_i))$, each $f_i$ occurring in
\eqref{eq:xdplusyd} must be a power of $x$ or $y$, and at least one of each must occur.  
This means that
$I_d = R[x,y]_d$ for $d \gg 0$, and so the Hilbert function of $I$
equals zero for $d \gg 0$, which contradicts the fact that the Hilbert
functions of $I$ and $J$ agree. 
\end{example}

\begin{theorem}\label{p:ascendingchain}
({Ascending chain condition.})
If $I, J \subseteq \R[x_0,\dotsc,x_n]$ are homogeneous tropical ideals with $I
\subseteq J$ and identical Hilbert functions then $I=J$.  
Thus there is no infinite ascending chain 
$I_1 \subsetneq I_2 \subsetneq I_3 \subsetneq \dotsb$ of tropical ideals
in $\R[x_0,\dotsc,x_n]$.
\end{theorem}
\begin{proof}
Let $\varphi : R[x_0,\dotsc,x_n] \to \mathbb B[x_0, \dotsc, x_n]$ 
be the semiring homomorphism 
described in the proof of Proposition \ref{prop:eventuallypoly}.
If $I \subseteq J$ are homogeneous tropical ideals in $\R[x_0,\dotsc,x_n]$
then $\varphi(I) \subseteq \varphi(J)$ are homogeneous tropical ideals
in $\mathbb B[x_0,\dotsc, x_n]$. Moreover, as $\varphi$ does not alter the
underlying matroids, if $I$ and $J$ have the same Hilbert function
then so do $\varphi(I)$ and $\varphi(J)$. In this case, for any $d \geq 0$
the matroids $M_d(\varphi(I))$ and $M_d(\varphi(J))$ have the same rank,
and any circuit of $M_d(\varphi(I))$ is a union of circuits of $M_d(\varphi(J))$,
so $M_d(\varphi(J))$ is a quotient of $M_d(\varphi(I))$; 
see \cite{Oxley}*{Proposition 7.3.6}. This implies that 
$M_d(\varphi(I)) = M_d(\varphi(J))$ \cite{Oxley}*{Corollary 7.3.4}. 
Given that $I_d \subseteq J_d$ and that the circuits of a
valuated matroid are determined (up to scaling) by their support,
we conclude that $\M_d(I) = \M_d(J)$. It follows that $I=J$.

Now, suppose that $I_1 \subseteq I_2 \subseteq I_3 \subseteq \dotsb$
is an infinite chain of homogeneous tropical ideals in
$R[x_0,\dotsc,x_n]$.  This gives rise to the chain $\varphi(I_1)
\subseteq \varphi(I_2) \subseteq \varphi(I_3) \subseteq \dotsb$ of
homogeneous tropical ideals in $\mathbb B[x_0, \dotsc, x_n]$.  By
Lemma \ref{l:monomial}, we can choose $\mathbf{w} \in \mathbb
R^{n+1}$ such that $\inn_{\mathbf w} (\varphi(I_j))$ is generated by
monomials for each $j$.  The chain $\inn_{\mathbf w}( \varphi(I_1))
\subseteq \inn_{\mathbf w}( \varphi(I_2 )) \subseteq \inn_{\mathbf w}(
\varphi(I_3 ))\subseteq \dots$ is then an infinite chain of monomial
ideals, which can then be regarded as ideals in $K[x_0,\dots,x_n]$ where
$K$ is any field. As $K[x_0,\dots,x_n]$ is Noetherian, this chain
must stabilize.  Since $\inn_{\mathbf w} (\varphi(I_j))$ and $I_j$ have
the same Hilbert functions, it follows that for large enough $j$ the
Hilbert functions of the $I_j$ are all equal. Our previous claim then
shows that for large enough $j$, the ideals $I_j$ are all the same, and so
the chain stabilizes.
\end{proof}

\begin{example}({\em Tropical ideals are not determined in low degrees.})
\label{e:infinitefamily}
We present an infinite family of distinct homogeneous tropical ideals
$(I_m')_{m > 0}$ in $\Rbar[x,y,z,w]$, all having the same Hilbert
function, such that for any $d \geq 0$, if $k, l \geq d$ then the
tropical ideals $I'_k$ and $I'_l$ agree on all their homogeneous parts
of degree at most $d$, i.e., $(I'_k)_i = (I'_l)_i$ for all $i \leq d$.
It follows that there is no bound $D$ depending only on the Hilbert
function of a tropical ideal $J$ for which the homogeneous parts
$(J_i)_{i \leq D}$ of degree at most $D$ determine the whole tropical
ideal $J$.  This contrasts with the case of ideals in a standard
polynomial ring, where the Gotzmann number \cite{BrunsHerzog}*{\S4.3} of
a Hilbert polynomial is a bound on the degrees of the generators for
any saturated ideal with that Hilbert polynomial.

For $\lambda \in \mathbb C$, consider the ideal 
$J_\lambda := \langle x-z-w,\, y-z-\lambda w \rangle$ in $\mathbb C [x,y,z,w]$,
and let $I_\lambda := \trop(J_\lambda)$. 
The Hilbert function of $I_\lambda$ is the same as that of $J_\lambda$, 
so $H_{I_\lambda}(d) = d+1$. 
We claim that for any degree $d \geq 0$,
the degree-$d$ parts of the homogeneous tropical ideals $I_\lambda$ are the same for all
but finitely many values of $\lambda$. 
Indeed, as we are working with the trivial valuation $\val: \mathbb C \to \Rbar$, 
the degree-$d$ part of $I_\lambda$ is determined by its rank-$(d+1)$
underlying matroid $\underline{\M_d(I_\lambda)}$, 
which encodes the linear dependencies
in the vector space $(\mathbb C [x,y,z,w] / J_\lambda)_d$ 
among the monomials of degree $d$.
The polynomials 
$x-z-w$ and $y-z-\lambda w$ form a Gr\"obner basis for $J_\lambda$
with respect to the lexicographic monomial order with $x > y > z > w$,
and the initial ideal $\inn(J_\lambda)$ is $\langle x,y \rangle$.
The corresponding standard monomials are the monomials that use only the
variables $z, w$.  These monomials form a basis for 
$\mathbb C [x,y,z,w] / J_\lambda$.
Consider the matrix $A$ whose columns are indexed by all monomials of degree $d$
and whose rows are indexed by the $d+1$ standard monomials of degree $d$,
where the entries of a column indexed by a monomial $\mathbf{x}^{\mathbf u}$ correspond 
to the unique way of expressing $\mathbf{x}^{\mathbf u}$ in terms of standard monomials
in $(\mathbb C [x,y,z,w] / J_\lambda)_d$. Note that the entries of $A$ are 
polynomials in $\lambda$. The bases of the matroid $\underline{\M_d(I_\lambda)}$ thus
correspond to the $(d+1)$-subsets of the columns of $A$ for which the 
associated maximal minor is nonzero. The collection of maximal minors 
of $A$ is a finite set of polynomials in $\lambda$, and so for any $\lambda$
that is not a root of any of these polynomials 
(except for the minors that are identically zero)
the matroid $\underline{\M_d(I_\lambda)}$ is the same, proving our claim.
We will call this matroid the generic matroid $M_d$.

We now show that that for any positive integer 
$n$, the matroid $\underline{\M_n(I_n)}$ is not the generic matroid $M_n$.
Consider the $n+1$ monomials 
$x^n$, $y z^{n-1}$, $z^{n-2} w^2$, $z^{n-3} w^3, \dotsc$,  $w^n$ in $\mon_d$.
These can be expressed in terms of standard monomials as 
$x^n = (z+w)^n$, $yz^{n-1} = z^n + \lambda z^{n-1}w$,
$z^{n-2} w^2$, $z^{n-3} w^3, \dotsc$, $w^n$. But then a simple linear
algebra computation shows that these elements are dependent only when
$\lambda = n$, as desired.

We define our family of tropical ideals $(I'_m)_{m > 0}$ inductively as follows.
Let $N_0 = 1$, and suppose we have defined the tropical ideals $I'_k$ for all $k 
< m$. Define $I'_m := I_{N_m}$, where $N_m>N_{m-1}$ is a large enough
integer such that for any degree $d \leq N_{m-1}$, 
the degree-$d$ part of 
$I_{N_m}$ is the generic matroid $M_d$. 
Since the matroid $\underline{\M_{N_m}(I_{N_m})}$ is not $M_{N_m}$,
all the tropical ideals $I_{N_m}$ defined in this way must be
distinct.  Also, by induction, we must have $N_{m-1} \geq m$.  We
conclude that the family $(I'_m)_{m > 0}$ satisfies the required
conditions: for any $d \geq 0$, if $k, l \geq d$ then the homogeneous
parts of $I'_k = I_{N_k}$ and $I'_l = I_{N_l}$ of degree at most $d$
are the generic matroids, as $N_{k-1}$ and $N_{l-1}$ are both
greater than $d$.
\end{example}

\section{Subschemes of tropical toric varieties}

\label{s:4}
We now extend the definition of tropical ideals to cover subschemes of
tropical toric varieties other than tropical projective space.

Tropical toric varieties are defined analogously to classical toric
varieties.  Given a rational polyhedral fan $\Sigma$, we associate an
affine tropical toric variety to each cone, and glue these together.
Explicitly, fix a lattice $N \cong \mathbb Z^n$ with dual lattice 
$M := \Hom(N,\mathbb Z) \cong \mathbb Z^n$.  Given a rational
polyhedral fan $\Sigma$ in $N_{\mathbb R} := N \otimes \mathbb R \cong \mathbb R^n$, 
to each $\sigma \in \Sigma$ we associate the vector space $N(\sigma) := N_{\mathbb
R}/\spann(\sigma)$.  The tropical toric variety associated to $\Sigma $ is
\[ \trop(X_{\Sigma}) := \textstyle \coprod_{\sigma \in \Sigma} N(\sigma).\]
To each $\sigma \in \Sigma$ we also associate the space
 $U_{\sigma}^{\trop} := \Hom(\sigma^{\vee} \cap M, \Rbar)$ of semigroup
homomorphisms to the semigroup $(\Rbar,\ttimes)$.
We have $U_{\sigma}^{\trop} = \coprod_{\tau \preceq \sigma} N(\tau)$,
where the union is over all faces $\tau$ of $\sigma$. 
The topology on $\Rbar$ has as basic open sets all open
intervals in $\mathbb R$, plus intervals of the form $(a,\infty] = \{
  b \in \mathbb R: b >a \} \cup \{\infty \}$. This induces the product
  topology on $\Rbar^{\sigma^{\vee} \cap M}$, and thus on
  $U_{\sigma}^{\trop}$.
See \cite{Kajiwara}, \cite{Payne}, \cite{Rabinoff} or
\cite{TropicalBook}*{Chapter 6} for more details on this construction.
Important special cases are the tropical versions of affine space,
projective space, and the torus; see Example~\ref{e:AnPn}.

One can also tropicalize the notion of Cox ring of a toric
variety.  In the case that the fan $\Sigma$ does not span all of
$N_{\mathbb R}$, for ease of notation we tropicalize the version 
given in \cite{CLS}*{(5.1.10)}.  The more invariant choice given in
\cite{CLS}*{Theorem 5.1.17} can also be tropicalized.

In what follows we will only consider rational polyhedral fans,
so will refer to them simply as fans. 

\begin{definition}({\em Cox semirings of tropical toric varieties.})
Let $\trop(X_{\Sigma})$ be a tropical toric variety defined by a fan
$\Sigma \subseteq N_{\mathbb R} \cong \mathbb R^n$.  Let $s$ be the
number of rays of $\Sigma$, and let $t = n - \dim(\spann(\Sigma))$.
Write $m = s+t$.  The {\em Cox semiring} of $\trop(X_{\Sigma})$ is the
semiring 
\[\Cox(\trop(X_{\Sigma})) := \Rbar[x_1,\dots,x_m].\]
Write $N$ as the direct sum $N'\oplus N''$, where the rays of $\Sigma$ 
span $N'_{\mathbb R}$.
Fix an $m \times n$ matrix  $Q$ whose $i$th row is the first
lattice point $\mathbf{v}_i$ on the $i$th ray of $\Sigma$ for  $1 \leq
i \leq s$, and whose last $t$ rows form a basis for $N''$.
We grade $\Cox(\trop(X_{\Sigma}))$ by the ``combinatorial Chow group''
$A_{n-1}(\Sigma) := \coker(M \cong \mathbb Z^n
\stackrel{\varphi}{\rightarrow} \mathbb Z^{m}),$ where the map $\varphi$
is given by $\varphi(\mathbf{u}) = Q \mathbf{u}$.  This gives the
last $t$ variables of $\Cox(\trop(X_{\Sigma}))$ degree $\mathbf{0} \in
A_{n-1}(\Sigma)$.

We will identify the rays of $\Sigma$ with the elements of $\{1,\dotsc,s\}$, and
thus general cones of $\Sigma$ with subsets of $\{1,\dotsc,s\}$.
For a cone $\sigma \in \Sigma$ write
\[\textstyle \mathbf{x}^{\hat{\sigma}} = \prod_{i \not \in \sigma} x_i \, \in \Rbar[x_1, \dotsc, x_m].\]  
The localization of the semiring $\Rbar[x_1,\dots,x_m]$ at a
monomial is defined analogously to the ring case.
The degree-zero part of the localization of the Cox semiring at
$\mathbf{x}^{\hat{\sigma}}$ is isomorphic to the semigroup semiring
$$(\Cox(\trop(X_{\Sigma}))_{\mathbf x^{\hat{\sigma}}})_{\mathbf{0}} 
\cong \Rbar[\sigma^{\vee} \cap M].$$  
Note that the last $t$ variables are always inverted in these localizations.
\end{definition}

Classical toric varieties can alternatively be described by the quotient construction
$X_{\Sigma} = (\mathbb A^m  \setminus V(B_\Sigma))/H_\Sigma$, where
$B_\Sigma$ denotes the
irrelevant ideal $\langle \mathbf x^{\hat{\sigma}} : \sigma \in \Sigma \rangle \subseteq K[x_1,\dots,x_m]$, 
and $H_\Sigma = \Hom(A_{n-1}(X_{\Sigma}),K^*)$.  
This construction also tropicalizes \cite{TropicalBook}*{Proposition 6.2.6}:
\begin{equation} \label{eqtn:toricquotient} \trop(X_{\Sigma}) = (\Rbar^m \setminus V(\trop(B_\Sigma)))/\trop(H_\Sigma),
\end{equation}
where $\trop(B_\Sigma)$ denotes the monomial ideal in $\Rbar[x_1,\dots,x_{m}]$ given by
\[ \trop(B_\Sigma) = \langle \mathbf x^{\hat{\sigma}} : \sigma \in \Sigma \rangle\]
and $\trop(H_\Sigma)$ is the $n$-dimensional subspace 
$\ker(Q^T) \subseteq \mathbb R^m$.  Write
\[\Rbar^{m}_{\sigma} :=
\{\mathbf{w} \in \Rbar^{m} : w_i = \infty \text{ for } i \in \sigma,
w_i \neq \infty \text{ for } i \notin \sigma \}.\]
The equality in
\eqref{eqtn:toricquotient} identifies $U_{\sigma}^{\trop}$ with
the subset $\bigl(\Rbar^{m} \setminus V(\mathbf
x^{\hat{\sigma}})\bigr)/\ker(Q^T)$, and $N(\sigma)$ with
$\bigl(\Rbar^{m}_{\sigma} \setminus V(\mathbf
x^{\hat{\sigma}})\bigr) / \ker(Q^T)$.  The topology on
$\trop(X_{\Sigma})$ is consistent with the quotient topology coming
from the topology on $\Rbar^m$.

\begin{example} \label{e:AnPn}
({\em Tropical affine space, projective space, and the torus.})
The fan $\Sigma$ consisting of the positive orthant in $N_{\mathbb R}
\cong \mathbb R^n$ gives rise to the toric
variety $\mathbb A^n$.  The tropical toric variety $\trop(\mathbb
A^n)$ is $\Rbar^n$.  Any face of the positive orthant has the form 
$\cone \{\mathbf{e}_i: i \in \sigma\}$ for $\sigma \subseteq \{1,\dotsc, n\}$; 
the stratum $N(\sigma)$ corresponds
precisely to the set $\Rbar^n_\sigma \subseteq \Rbar^n$.
The Cox semiring $\Cox(\trop(\mathbb A^n))$
is equal to $\Rbar[x_1,\dots,x_n]$ with the trivial grading $\deg(x_i)=0$
for all $i$.

The fan $\Sigma$ defining $\mathbb P^n$ has rays spanned by
$\mathbf{e}_1,\dots,\mathbf{e}_n$ and $\mathbf{e}_0=-\sum_{i=1}^n
\mathbf{e}_i$ in $N_{\mathbb R} \cong \mathbb R^n$, and cones spanned
by any $j$ of these rays for $j \leq n$. 
The Cox semiring in this case is $\Rbar[x_0,\dots,x_n]$ with the
standard grading $\deg(x_i)=1$ for all $i$. The tropicalization
$\trop(\mathbb P^n)$ equals $(\Rbar^{n+1} \setminus
\{(\infty,\dots,\infty)\})/\mathbb R \mathbf{1}$, where $\mathbf{1} =
(1,\dots,1)$. Figure \ref{f:toricprojective} shows the fan
$\Sigma$ and its corresponding tropical toric variety 
$\trop(\mathbb P^2) = \coprod_{\sigma \in \Sigma} N(\sigma)$ in the case
where $n=2$.
\begin{figure}[htb]
 \begin{center}
  \includegraphics[scale=1]{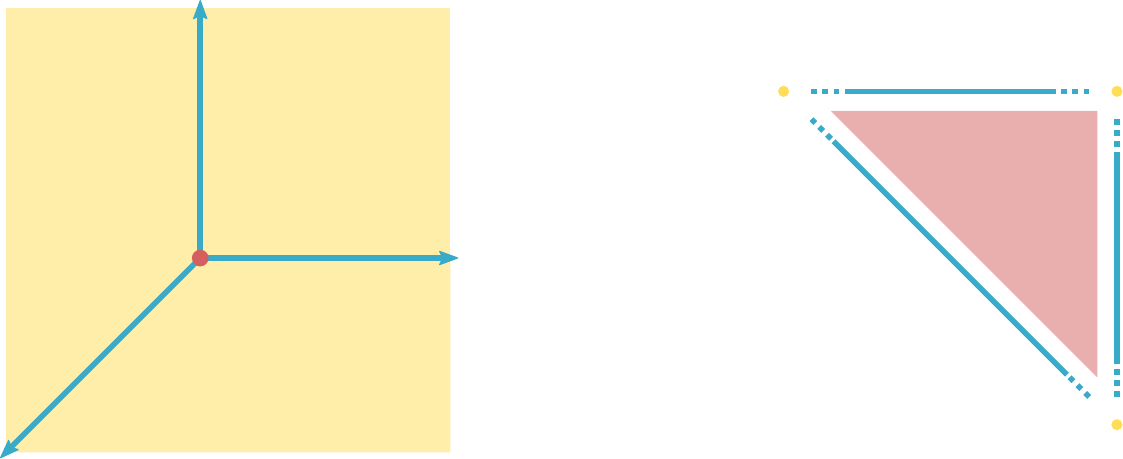}
  \caption{\label{f:toricprojective} 
  Tropical projective space $\trop(\mathbb P^2)$ as a tropical toric variety.}
 \end{center}
\end{figure} 

The fan $\Sigma$ consisting of just the origin in $N_{\mathbb R} \cong
\mathbb R^n$ gives rise to the tropical torus $\mathbb R^n$. Its Cox
semiring is $\Rbar[x_1,\dots,x_n]$ with the trivial grading
$\deg(x_i)=0$ for all $i$.  The localization at the unique cone
$\sigma = \{\mathbf 0\}$ is $\Rbar[x_1,\dots,x_n]_{x_1 \dotsb x_n} \cong
\Rbar[x_1^{\pm 1},\dots,x_n^{\pm 1}]$.  
\end{example}

Subschemes of a classical toric variety can be described by ideals in
its Cox ring.  We now extend this to the tropical realm.

\begin{definition}({\em Tropical ideals in Cox semirings.})
\label{d:tropicalidealtoric}
Let $\Sigma$ be a fan in $N_{\mathbb R}$, and
fix a cone $\sigma \in \Sigma$.  
An ideal $I \subseteq \Rbar[\sigma^{\vee} \cap M]$ is
a {\em tropical ideal} if for every finite set $E$ of monomials in
$\Rbar[\sigma^{\vee} \cap M]$, the set of polynomials with support in
$E$ is the collection of vectors of a valuated matroid on $E$.

Set $S = \Cox(\trop(X_{\Sigma}))$.  A homogeneous ideal $I \subseteq
S$ is a {\em locally tropical ideal} if for every $\sigma \in \Sigma$
the ideal $(IS_{\mathbf x^{\hat{\sigma}}})_{\mathbf 0} \subseteq 
\Rbar[\sigma^{\vee} \cap M]$ is a tropical ideal.

Suppose the rays $\Sigma(1)$ of $\Sigma$ positively span $N_{\mathbb
  R}$, so the grading of $\Cox(\trop(X_{\Sigma}))$ is positive and for
any degree $\mathbf b$ the homogeneous graded piece $I_{\mathbf b}$
contains finitely many monomials (see \cite{MillerSturmfels}*{Chapter
  8}). In this case, a homogeneous ideal $I \subseteq
\Cox(\trop(X_{\Sigma}))$ is a {\em homogeneous tropical ideal} if
every graded piece $I_{\mathbf{b}}$ is the set of vectors of a
valuated matroid on the set of monomials of degree $\mathbf{b}$.
\end{definition}

\begin{proposition}\label{p:tropicalidealsareweak}
Let $\Sigma$ be a fan whose rays positively span $N_{\mathbb R}$.  
Any homogeneous tropical ideal in $\Cox(\trop(X_{\Sigma}))$ is a locally tropical 
ideal.
\end{proposition}
\begin{proof}
Write $S = \Cox(\trop(X_{\Sigma}))$, and suppose that $I \subseteq S$ is a homogeneous tropical
ideal.  Fix $\sigma \in \Sigma$, and let $E$ be a
collection of monomials in
$(S_{\mathbf{x}^{\hat{\sigma}}})_{\mathbf{0}}$. 
All monomials in $E$ have the
form $\mathbf{x}^{\mathbf{u}}/(\mathbf{x}^{\hat{\sigma}})^l$ for some $\mathbf{u} \in \mathbb N^s$
and $l \geq 0$.  Choosing a common denominator, we may
assume that the exponent $l$ is the same for all monomials in $E$.
This means that for all $f \in
(IS_{\mathbf{x}^{\hat{\sigma}}})_{\mathbf{0}}$ with support in $E$ we have 
$g := (\mathbf{x}^{\hat{\sigma}})^lf \in
IS_{\mathbf{x}^{\hat{\sigma}}} \cap S = (I: (\mathbf{x}^{\hat{\sigma}})^{\infty})$. 
Note that $g$ is homogeneous of degree
$l\deg(\mathbf{x}^{\hat{\sigma}})$. Moreover, given any homogeneous polynomial 
$g$ of degree $l\deg(\mathbf{x}^{\hat{\sigma}})$
in $(I:(\mathbf{x}^{\hat{\sigma}})^{\infty})$, there is some $l' \geq 0$ for which 
$(\mathbf{x}^{\hat{\sigma}})^{l'}g \in
I$, so $g/(\mathbf{x}^{\hat{\sigma}})^l \in (IS_{\mathbf{x}^{\hat{\sigma}}})_{\mathbf{0}}$.
This gives a bijection between polynomials in
$(IS_{\mathbf{x}^{\hat{\sigma}}})_{\mathbf{0}}$ with support in $E$
and homogeneous polynomials in $(I:(\mathbf{x}^{\hat{\sigma}})^{\infty})$ 
of degree $l\deg(\mathbf{x}^{\hat{\sigma}})$ with
support in $(\mathbf{x}^{\hat{\sigma}})^lE$. 

Since the restriction of the set of vectors of a valuated matroid 
to a subset of its ground set is also a valuated matroid, 
it thus suffices to show that the set
of polynomials in $(I:(\mathbf{x}^{\hat{\sigma}})^{\infty})$ of degree 
$l\deg(\mathbf{x}^{\hat{\sigma}})$ is the set of
vectors of a valuated matroid.  This follows from the fact that if $I$
is a homogeneous tropical ideal then $J :=(I:(\mathbf{x}^{\mathbf{u}})^{\infty})$ is
as well for any monomial $\mathbf{x}^{\mathbf{u}}$.  Indeed, suppose
that $f, g \in J_{\mathbf{b}}$.  Then there is $l \geq 0$ for which
$\mathbf{x}^{l \mathbf{u}}f, \mathbf{x}^{l \mathbf{u}}g \in I$, and thus the
vector elimination axiom for $J_{\mathbf{b}}$ follows from the vector
elimination axiom for $I_{\mathbf{b}+l\deg(\mathbf{x}^{\mathbf{u}})}$.
\end{proof}

\begin{example}
Let $\Sigma$ be the fan in $N_{\mathbb R}$ consisting of just one cone
$\sigma = \{\mathbf 0\}$, so the corresponding
tropical toric variety is the tropical torus $\mathbb R^n$. An ideal
$I \subseteq \Cox(\trop(X_{\Sigma}))$ is a locally tropical ideal if
for every finite set $E$ of monomials in
$(\Cox(\trop(X_{\Sigma}))_{\mathbf{x}^{\hat{\sigma}}})_{\mathbf{0}}
\cong \Rbar[x_1^{\pm 1},\dots,x_n^{\pm 1}]$, the set of polynomials in
the image of $I$ with support in $E$ is the set of vectors of a
valuated matroid on $E$.

If $\Sigma$ is the positive orthant in $N_{\mathbb R}$, so the
corresponding tropical toric variety is $\Rbar^n$, the condition
for an ideal $I \subseteq \Cox(X_{\Sigma}) \cong \Rbar[x_1,\dots,x_n]$ to be
a tropical ideal (and a locally tropical ideal) 
is analogous: for every finite set $E$ of monomials
in $\Rbar[x_1,\dots,x_n]$, the set of polynomials in $I$ with support in $E$
should be the set of vectors of a valuated matroid on $E$.
Note that this agrees with Definition \ref{d:tropicalideal} in the introduction.
\end{example}

For a tropical polynomial $f = \sum a_{\mathbf{u}} \ttimes
\mathbf{x}^{\mathbf{u}} \in \Rbar[x_1,\dots,x_n]$, the
\emph{homogenization} of $f$ is the polynomial $\tilde{f} := \sum
a_{\mathbf{u}} \ttimes x_0^{d-|\mathbf{u}|} \ttimes
\mathbf{x}^{\mathbf{u}} \in \Rbar[x_0,x_1,\dots,x_n]$, where $d =
\max_{a_{\mathbf{u}} \neq \infty} |\mathbf{u}|$.  The homogenization
of an ideal $I \subseteq \Rbar[x_1,\dots,x_n]$ is $\tilde{I} := \langle
\tilde{f} : f \in I \rangle$.  This has the property that every
homogeneous polynomial $g \in \tilde{I}$ has the form $x_0^m \ttimes
\tilde{f}$ for some $m \geq 0$ and $f \in I$, and so the set of homogeneous
polynomials of degree $d$ in $\tilde{I}$ is in bijection with the set
of polynomials in $I$ of degree at most $d$.  This implies that if $I
\subseteq \Rbar[x_1,\dots,x_n]$ is a (locally) tropical ideal
then $\tilde{I}$ is a homogeneous tropical ideal.

\begin{example}({\em The affine ideal of a point.})\label{e:affineidealofpoint}
For any $\mathbf a \in \Rbar^n$, let $J_{\mathbf a} \subseteq
\Rbar[x_1,\dots,x_n]$ be the ideal consisting of all tropical
polynomials $f$ whose variety $V(f)$ contains $\mathbf a$.  Its
homogenization $\tilde{J_{\mathbf a}} \subseteq
\Rbar[x_0,x_1,\dots,x_n]$ is the homogeneous ideal $J_{(0,\mathbf a)}$
of the point $(0,\mathbf a) \in \Rbar^{n+1}$ (see Example
\ref{e:idealofpoint}). By Lemma
\ref{p:tropicalidealsareweak} $J_{\mathbf a}$ is a tropical ideal, as
$(\Rbar[x_0,\dots,x_n]_{x_0})_{\mathbf{0}} \cong
\Rbar[x_1,\dots,x_n]$, with the isomorphism taking $x_0$ to $0$.  We
will see in Example \ref{e:maximalideals} that the ideals of the form
$J_{\mathbf a}$ are precisely the maximal tropical ideals of
$\Rbar[x_1,\dots,x_n]$.
\end{example}

The following example shows that the converse of Proposition \ref{p:tropicalidealsareweak}
is not true.

\begin{example}({\em Locally tropical ideals might not be homogeneous tropical ideals.})
Let $X_{\Sigma} = \mathbb P^1$, so $S := \Cox(\trop(X_{\Sigma})) =
\Rbar[x,y]$ with the standard grading.  Let $I := \langle
x \! \tplus \! y,x^3,y^3 \rangle$.  Then $(IS_{x})_0 = (S_x)_0 = \Rbar[y/x]$, 
$(IS_{y})_0 = (S_y)_0 = \Rbar[x/y]$, and $(IS_{xy})_0
= (S_{xy})_0 = \Rbar[(x/y)^{\pm 1}]$. Since these localizations are all equal 
to their respective coordinate semirings, they are trivially
tropical ideals.  However, 
$I_2$ is the $\Rbar$-semimodule spanned by $x^2\tplus xy$ and $xy \tplus
y^2$, which is not the set of vectors of a valuated matroid, as the polynomial 
$x^2 \tplus y^2$ is not in $I_2$.
\end{example}

\begin{remark}
The notion of being a (locally) tropical ideal is invariant under automorphisms
of the coordinate semirings, as the lack of cancellation forces any automorphism 
to send monomials into monomials.
In the case that the toric variety is the torus, for example, 
the coordinate semiring $\Rbar[x_1^{\pm 1},\dots,x_n^{\pm 1}]$ 
has a natural action of $\GL(n,\mathbb Z)$
by monomial change of coordinates: for any matrix $A \in \GL(n,\mathbb
Z)$ we have $A \cdot \mathbf x^{\mathbf u} = \mathbf x^{A \mathbf u}$.
If $I$ is a tropical ideal in $\Rbar[x_1^{\pm 1},\dots,x_n^{\pm 1}]$
then $A \cdot I$ is as well, as the $\GL(n,\mathbb Z)$ action permutes
finite sets of monomials.
\end{remark}

\begin{remark}
The ascending chain condition (Theorem~\ref{p:ascendingchain}) holds
for homogeneous tropical ideals in the positively-graded case, with the same proof,
but does not hold verbatim for locally tropical ideals.  For example,
let $I_n := \langle x^2, y^2, x \tplus i \ttimes y : 0 \leq i \leq n
\rangle \subseteq \Rbar[x,y] = \Cox(\trop(\mathbb P^1))$.  Then $\{
I_n \}_{n \geq 0}$ is an infinite ascending chain of locally tropical
ideals. However, in this case, for all cones $\sigma$ in the fan of $\mathbb
P^1$ the ideals $(I\Rbar[x,y]_{\mathbf{x}^{\hat{\sigma}}})_\mathbf{0}$ are the
unit ideal $\langle 0 \rangle$.  
All the $I_n$ thus correspond geometrically to the same
(empty) subscheme of $\trop(\mathbb P^1)$. 

For a general tropical
toric variety, one can define an equivalence relation on the set of locally
tropical ideals in $S:=\Cox(\trop(X_{\Sigma}))$ by setting $I \sim J$
if $(IS_{\mathbf{x}^{\hat{\sigma}}})_{\mathbf{0}} =
(JS_{\mathbf{x}^{\hat{\sigma}}})_{\mathbf{0}}$ for all $\sigma \in \Sigma$. 
The proof of Theorem~\ref{p:ascendingchain} then shows that there is no infinite
ascending chain of non-equivalent locally tropical ideals.
\end{remark}

When working with semirings, {congruences} perform some of the
functions of ideals in rings.  A {\em congruence} on a semiring $S$ is an
equivalence relation $\sim$ on $S$ that is compatible with addition and
multiplication: $f \sim g$ and $f' \sim g'$ imply that $(f \tplus f')
\sim (g \tplus g')$ and $(f \ttimes g) \sim (g \ttimes g')$.
In \cite{Giansiracusa2}, Jeffrey and Noah Giansiracusa defined a congruence 
associated to an ideal in the tropical setting, as we now recall.

\begin{definition}({\em Bend relations \cite{Giansiracusa2}.})
Let $S := \Rbar[\sigma^{\vee} \cap M]$ for a cone $\sigma \subseteq
N_{\mathbb R}$.  For a polynomial $f = \bigoplus a_{\mathbf u} \ttimes \mathbf
x^{\mathbf u} \in S$ denote by
$f_{\hat{\mathbf v}}$ the polynomial obtained by deleting the term
involving $\mathbf x^{\mathbf v}$ from $f$, i.e., $f_{\hat{\mathbf v}}
= \bigoplus_{\mathbf u \neq \mathbf v} a_{\mathbf u} \ttimes \mathbf
x^{\mathbf u}$.  The {\em bend relations} of $f$ are
\[B(f) := \{ f \sim f_{\hat{\mathbf v}} : \mathbf x^{\mathbf v} \in \supp(f) \}.\]
The {\em bend congruence} $\mathcal B(I)$ of an ideal $I \subseteq S$ is the congruence on $S$
\[\mathcal B(I) := \langle B(f) : f\in I \rangle .\]
The coordinate semiring associated to $I$ is $S/\mathcal B(I)$.
\end{definition}

This construction is motivated by the fact that 
for any ideal $I \subseteq \Rbar[x_1,\dots,x_n]$, there is a
canonical isomorphism $\Hom(S/\mathcal B(I),\Rbar) \cong V(I) \, \subseteq \Rbar^n$.

In \cite{Giansiracusa2}, the tropicalization of the subscheme of
$\mathbb A^n$ defined by an ideal $J \subseteq K[x_1,\dots,x_n]$ is
defined to be $\Sxpec( \Rbar[x_1,\dots,x_n]/\mathcal B(\trop(J)))$.
Here $\Sxpec$ is used in sense of $\mathbb F_1$-geometry, so it denotes a
topological space with a sheaf of semirings.  The underlying space of
$\Sxpec(R)$ for a semiring $R$ is the set of prime ideals of $R$,
where an ideal $I$ is {\em prime} if $f \ttimes g \in I$ implies that $f \in I$
or $g \in I$.  Similarly, the tropicalization of the subscheme of
$\mathbb P^n$ defined by a homogeneous ideal $J \subseteq K[x_0,\dots,x_n]$ is
$\Proj(\Rbar[x_0,\dots,x_n]/\mathcal B(\trop(J)))$, where the underlying set
of $\Proj$ is the set of homogeneous prime ideals not containing
$\langle x_0,\dots,x_n \rangle$.  This is explained in 
\cite{LorscheidBlueSchemesRelative}*{\S 9}, where these are presented as a
special case of Lorscheid's blueprints.  See also
\cite{LorscheidBluePrints1}, \cite{LorscheidSchemetheoretic},
\cite{Durov}, and \cite{ToenVaquie} for related work.  The results in
\cite{Giansiracusa2} are set in this context, so they describe the
tropicalization of any subscheme of a classical toric variety.  We now
generalize this construction to include more general subschemes of
tropical toric varieties that are not necessarily tropicalizations.

\begin{definition}  \label{d:tropicalsubscheme} ({\em Subschemes of tropical 
toric varieties.})
Let $I$ be a locally tropical ideal of the Cox semiring $S$ of a tropical
toric variety $\trop(X_{\Sigma})$.  The subscheme of $\trop(X_{\Sigma})$ defined 
by $I$ is 
given locally as a subscheme of $U_{\sigma}^{\trop}$ by 
\[\Sxpec(\Rbar[\sigma^{\vee} \cap M]/\mathcal B((IS_{\mathbf 
x^{\hat{\sigma}}})_{\mathbf{0}})).\]
\end{definition}

Definition~\ref{d:tropicalsubscheme} makes sense without the
restriction that $I$ be a locally tropical ideal; 
the reason for requiring this condition is that, as shown in Example~\ref{e:arbitrary},
without it the variety of the subscheme
might not be a finite polyhedral complex.

\begin{remark}\label{rem:tropicalspectra}({\em Tropical spectra.})
In Definition~\ref{d:tropicalsubscheme}, 
instead of using the whole collection $\Sxpec(R)$ of prime ideals 
of a semiring $R$, it might be more useful to just take 
the set of {\em tropical} prime ideals as the underlying topological space,
which could be called the {\em tropical spectrum} $\TSxpec(R)$ of $R$.
For instance, consider $R$ to be the quotient of $\Rbar[x]$ by the congruence that
identifies two polynomials if they correspond to the same function
$\Rbar \rightarrow \Rbar$, such as $x^2 \tplus 1$ and $x^2 \tplus 7
\ttimes x \tplus 1$.  In \cite{Giansiracusa2}*{Proposition 3.4.1},
it is shown that the points of $\Sxpec(R)$ are in bijection
with intervals of $\Rbar$; the points of $\Sxpec(\Rbar[x])$ are
thus at least this complicated.  In Example~\ref{e:tropicalprimeR[x]}  we
show that the set of {\em tropical} prime ideals of $\Rbar[x]$ 
reflects better the classical situation.
\end{remark}

\begin{example}({\em Tropical prime ideals of $\Rbar[x]$.}) 
\label{e:tropicalprimeR[x]}
We classify all tropical prime ideals in the semiring $\Rbar[x]$ of
tropical polynomials in one variable.  For any $a \in \Rbar$, let
$J_a$ be the tropical ideal consisting of all tropical polynomials $f$ whose
variety $V(f) \subseteq \Rbar$ contains $a$ (see Example \ref{e:affineidealofpoint}).
If $a = \infty$ then $J_a$ is the
ideal of polynomials with no constant term. 
The fact that any two tropical
polynomials $f, g$ satisfy $V(f \ttimes g) = V(f) \cup V(g)$ implies
that $J_a$ is a tropical prime ideal.  We will prove that, in fact,
these are the only proper nontrivial prime tropical ideals in $\Rbar[x]$.

Fix a prime tropical ideal $P \subseteq \Rbar[x]$, and let $f =
\bigoplus_{i=0}^d b_i \ttimes x^i \in P$, where $b_d \neq \infty$.
Let $V(f) = \{a_1,\dots,a_s\}$.  The multiplicity $m_i$ of $a_i$ is the
maximum of $j_2 -j_1$ where 
both $j_1, j_2$ achieve the minimum in $\min_j(b_j + ja_i)$.
Set $g := b_d \ttimes \prod_{i=1}^s (x \tplus
a_i)^{m_i}$.  The polynomial $g$ is the tropical polynomial with
smallest coefficients for which $f(z) = g(z)$ for all $z \in \Rbar$; 
see, for example, \cite{GriggManwaring}*{Section 4}. Writing $g
= \bigoplus_{j=0}^d c_j \ttimes x^j$, we have
\[
c_j = \min \,\, \{b_j\} \cup 
\left\{ \frac{b_i \cdot (k-j) + b_k \cdot (j-i)}{k-i} : 0 \leq i < j < k \leq d 
\right\}
\]
for any $j$, as stated in \cite{GriggManwaring}*{Lemma 3.3}.
In particular, $c_0=b_0$ and $c_d=b_d$.
The two consequences of this formula 
that we will use are that $c_i \leq b_i$ for all $i$, and if $c_i<b_i$
then there is $l$ with $c_{i-1}-a_l = c_i = c_{i+1}+a_l$. 

We claim that $g\ttimes f = g^2$, and so
$g^2 \in P$. To show this we need to prove that for all $0 \leq k \leq 2d$
we have $\min_{i+j=k}(b_i+c_j) = \min_{i+j=k}(c_i+c_j)$.  The
inequality $\geq$ follows from the fact that $b_i \geq c_i$ for all $i$.
Suppose now that there is a pair $(i,j)$ with $c_{i}+c_{j} = \min_{i'+j'=i+j}
(c_{i'}+c_{j'}) < \min_{i'+j'=i+j}(b_{i'}+c_{j'})$.  This implies that $c_i<b_i$ 
and
$c_j<b_j$, so $0<i,j<d$ and there must be $l,l'$ with
$c_{i-1}-a_l=c_i=c_{i+1}+a_l$ and $c_{j-1}-a_{l'}
=c_j=c_{j+1}+a_{l'}$.  Without loss of generality we may assume that
$a_l \leq a_{l'}$.  We have $c_i+c_j = c_{i-1}-a_l + c_{j+1}
+a_{l'}$, so $c_{i-1}+c_{j+1} \leq c_i+c_j$.  We may thus replace
$(i,j)$ by $(i-1,j+1)$ and repeat.
After a finite number of
iterations we will have a pair $(i,j)$ with $c_i+c_j$ achieving the
minimum $\min_{i'+j'=i+j} (c_{i'}+c_{j'})$ 
and at least one of $c_i=b_i$ or $c_j=b_j$,
which is a contradiction. This proves that $g^2=g \ttimes f$.

Since $P$ is prime, the fact that $g^2 \in P$ implies that $x \tplus
a \in P$ for some $a = a_i$.  Multiplying by powers of $x$ we see that
$P$ must contain all polynomials of the form $x^j \tplus a \ttimes
x^{j-1}$. Furthermore, the vector elimination axiom 
 forces $P$ to contain all polynomials of the form $a^m \ttimes x^l \tplus a^{l}\ttimes
x^m$. Since these polynomials generate the ideal $J_a$, 
we see that $J_a \subseteq P$.  If $P$ were strictly
larger than $J_a$, it would contain a polynomial $f'$ with $a \not
\in V(f')$.  The argument above then shows that $x \tplus b \in P$ for
some $b \in V(f')$.  But then the vector elimination axiom applied
to $x \tplus a$ and $x \tplus b$ forces $P$ to contain the constant
$\min (a, b)$, which implies that $P$ is the unit ideal.
\end{example}

We conclude this section by extending the notion of variety to ideals in 
a Cox semiring $\Cox(\trop(X_{\Sigma}))$. 
For a fan  $\Sigma \subseteq N_{\mathbb R} \cong \mathbb R^{n}$ 
and $\sigma \in \Sigma$, the coordinate ring of the corresponding
stratum $N(\sigma)$ is 
\[(S_{\mathbf{x}^{\hat{\sigma}}}/\langle x_i
\sim \infty : i \in \sigma \rangle)_{\mathbf{0}} \cong
\Rbar[\sigma^{\perp} \cap M] \cong \Rbar[y_1^{\pm
    1},\dots,y_{n-\dim(\sigma)}^{\pm 1}].\]  
The {\em variety} of an ideal $I \subseteq \Rbar[y_1^{\pm 1},\dots,y_{l}^{\pm 1}]$
is 
$$V(I) := \{ \mathbf{w} \in \mathbb R^l : \text{ the min in }
f(\mathbf{w}) \text{ is achieved at least twice for all } f \in I \setminus \{ \infty \}
\}.$$
For $f \in \Rbar[y_1^{\pm 1},\dots,y_l^{\pm 1}]$ the initial term
$\inn_{\mathbf{w}}(f)$ is defined as in Definition \ref{d:initialIdeal}.  Similarly,
for an ideal $I \subseteq \Rbar[y_1^{\pm 1},\dots,y_l^{\pm 1}]$, its initial
ideal is $\inn_{\mathbf{w}}(I) := \langle \inn_{\mathbf{w}}(f) : f \in
I \rangle \subseteq \mathbb B[y_1^{\pm 1},\dots,y_l^{\pm 1}]$.

\begin{definition}({\em Varieties of ideals in Cox semirings.})
\label{def:variety}
Fix a fan $\Sigma \subseteq N_{\mathbb R} \cong \mathbb R^{n}$, and
let $I$ be a homogeneous ideal in $S:=\Cox(\trop(X_{\Sigma}))$.
We define
$$V_{\sigma}(I) := V((I S_{\mathbf{x}^{\hat{\sigma}}}/\langle x_i \sim \infty : 
i
\in \sigma \rangle)_{\mathbf{0}}) \subseteq N(\sigma) \cong \mathbb 
R^{n-\dim(\sigma)},$$ 
and
$$V_{\Sigma}(I) := \coprod_{\sigma \in \Sigma} V_{\sigma}(I) \subseteq
\coprod N(\sigma) = \trop(X_{\Sigma}).$$
\end{definition}

When $\Sigma$ is the positive orthant in $\mathbb R^n$, the tropical
toric variety $\trop(X_{\Sigma})$ is $\trop(\mathbb A^n)
\cong \Rbar^{n}$, and the Cox semiring $S$ equals $\Rbar[x_1,\dots,x_n]$
with the trivial grading $\deg(x_i)=0$ for all $i$.  In the
next lemma we check that the new notion of variety agrees with that in
\eqref{eqtn:variety}, and also that the easy equivalences of the
Fundamental Theorem \cite{TropicalBook}*{Theorem 3.2.3} hold in this
new setting. 

\begin{proposition} \label{p:varietynonmonomial}
Let $\Sigma$ be the positive
orthant in $N_{\mathbb R}$,
and let $I$ be an ideal in $\Rbar[x_1,\dots,x_n] = \Cox(\trop(\mathbb
A^n))$.  Then
\begin{align*}
 V_{\Sigma}(I) = V(I) = \{ \mathbf{w} \in \Rbar^{n} \mid & \text{ for all } f 
\in I \text{ with } f(\mathbf{w})<\infty,  
 \text{ the minimum in } f(\mathbf{w}) \\
& \text{ is achieved at least twice} \}.
\end{align*}
Moreover, we have $\mathbf{w} \in V(I)$ if and only if
$\inn_{\mathbf{w}}(I)$ does not contain a monomial.  Similarly, if $I$
is an ideal in $\Rbar[x_1^{\pm 1},\dots,x_n^{\pm 1}]$ and $\mathbf{w}
\in \mathbb R^n$, then $\mathbf{w} \in V(I)$ if and only if
$\inn_{\mathbf{w}}(I) \neq  \langle 0 \rangle$.
\end{proposition}
\begin{proof}
Let $S=\Rbar[x_1,\dots,x_n]$.  To show the containment
$V_{\Sigma}(I) \subseteq V(I)$, fix $\mathbf{w} \in V_{\sigma}(I)
\subseteq N(\sigma) \subseteq \Rbar^n$ for some $\sigma \subseteq
\{1,\dots,n\}$.
Suppose $f \in
I$ satisfies $f(\mathbf{w}) < \infty$.  Write $f = f^{\sigma} + f'$, where
every term in $f'$ is divisible by some variable $x_i$ with $i \in
\sigma$ and no term in $f^{\sigma}$ is. In particular, $f^{\sigma}(\mathbf{w}) < \infty$ and
$f'(\mathbf{w}) = \infty$. We have $f^{\sigma} \in
(IS_{\mathbf{x}^{\hat{\sigma}}}/\langle x_i \sim \infty: i \in \sigma
\rangle)_{\mathbf{0}} \subseteq (S_{\mathbf{x}^{\hat{\sigma}}}/\langle
x_i \sim \infty: i \in \sigma \rangle)_{\mathbf{0}} \cong
\Rbar[x_i^{\pm 1} : i \not \in \sigma]$.  Since
$\mathbf{w} \in V_{\sigma}(I)$, the minimum in $f^{\sigma}(\mathbf{w})$ is
achieved at least twice, and so the same is true for the minimum in
$f(\mathbf{w})$.  This proves that $V_{\sigma}(I) \subseteq V(I)$.

Conversely, suppose that $\mathbf{w} \in \Rbar^{n}$ has the property
that the minimum in $f(\mathbf{w})$ is achieved at least twice
whenever $f \in I$ with $f(\mathbf{w})<\infty$.  Let $\sigma := \{ i :
w_i =\infty \}$.  For any $f \in I$ with $f(\mathbf{w}) < \infty$,
write $f =f^{\sigma}+f'$ as above.  We have $f(\mathbf{w}) = f^{\sigma}(\mathbf{w})$,
and the minimum in $f(\mathbf w)$ is achieved at at least two terms in $f^{\sigma}$.  As the images of
$f$ and $f^{\sigma}$ agree in $S_{\mathbf{x}^{\hat{\sigma}}}/\langle x_i \sim
\infty : i \in \sigma \rangle$, it follows that $\mathbf{w}
\in V_{\sigma}(I)$, and so $\mathbf{w} \in V_{\Sigma}(I)$.  This
completes the proof that $V(I) = V_{\Sigma}(I)$.

To prove the last assertion, note that if $I$ is an ideal in 
the semiring $\Rbar[x_1,\dotsc,x_n]$ (or $\Rbar[x_1^{\pm 1},\dotsc,x_n^{\pm 1}]$),
since there is no cancellation in $\mathbb B[x_0,\dots,x_n]$ (or $\mathbb B[x_1^{\pm 1},\dotsc,x_n^{\pm 1}]$),
the initial ideal $\inn_{\mathbf{w}}(I)$ contains a monomial if and only if there is $f \in I$ with
$\mathbf{x}^{\mathbf{u}} = \inn_{\mathbf{w}}(f)$.
 This happens precisely when $f(\mathbf{w}) < \infty$ and the
minimum in $f(\mathbf{w})$ is achieved only once, i.e., when $\mathbf{w} \not
\in V(I)$. 
\end{proof}

Subvarieties of tropical toric varieties can also be described 
in terms of the quotient construction \eqref{eqtn:toricquotient}.

\begin{lemma} \label{l:CoxConstructionVariety}
Fix a fan $\Sigma \subseteq N_{\mathbb R}$, and let $I$ be a homogeneous ideal in
$\Cox(\trop(X_{\Sigma}))$.  The variety of $I$ is
$$ V_{\Sigma}(I) =  (V(I) \setminus V(\trop(B_\Sigma)))/\ker(Q^T) \subseteq 
\trop(X_{\Sigma}).$$
In particular, if $\Sigma$ is the fan of $\mathbb P^n$ and 
$I$ is a homogeneous ideal in $\Rbar[x_0,\dots,x_n]$
with respect to the standard grading, then $V_\Sigma(I) \subseteq
\trop(\mathbb P^n)$ is equal to $(V(I) \setminus \{ (\infty,\dots,\infty)
\})/\mathbb R \mathbf{1}$.
\end{lemma}

\begin{proof}
 Write $s$ for the number of rays of $\Sigma$, $t = n -
 \dim(\spann(\Sigma))$, and $m=s+t$.  
Recall that the quotient construction \eqref{eqtn:toricquotient}
identifies the stratum  $N(\sigma)$ with $\bigl(\Rbar^{m}_{\sigma}
 \setminus V(\mathbf x^{\hat{\sigma}})\bigr) / \ker(Q^T)$.  By
Proposition~\ref{p:varietynonmonomial}, a vector $\mathbf{w} \in
\Rbar^{m}_{\sigma}$ has image $\overline{\mathbf{w}}$ in $V_{\sigma}(I)
\subseteq N(\sigma)$ if and only if $0 \notin
\inn_{\overline{\mathbf{w}}}((IS_{\mathbf{x}^{\hat{\sigma}}}/\langle x_i \sim
\infty : i \in \sigma \rangle)_{\mathbf{0}})$.

For any $\mathbf{w} \in \Rbar^m_{\sigma}$, the image in
 $(\mathbb
B[x_1,\dots,x_m]_{\mathbf{x}^{\hat{\sigma}}}/\langle x_i \sim \infty :
i \in \sigma \rangle)_{\mathbf{0}} \cong \mathbb B[y_1^{\pm
    1},\dots,y_{n-\dim(\sigma)}^{\pm 1}]$ of $\inn_{\mathbf{w}}(I)$ 
 equals
$\inn_{\overline{\mathbf{w}}}((IS_{\mathbf{x}^{\hat{\sigma}}}/\langle
x_i \sim \infty : i \in \sigma \rangle)_{\mathbf{0}})$.  Thus we have
$0 \in
\inn_{\overline{\mathbf{w}}}((IS_{\mathbf{x}^{\hat{\sigma}}}/\langle
x_i \sim \infty : i \in \sigma \rangle)_{\mathbf{0}})$ if and only if
some power of $\mathbf{x}^{\hat{\sigma}}$ lies in the ideal
$\inn_{\mathbf{w}}(I)\mathbb B[x_1,\dots,x_m]/\langle x_i \sim \infty
: i \in \sigma \rangle$.  Since $w_i = \infty$ for $i \in \sigma$,
this happens if and only if there is $f \in I$ with
$\inn_{\mathbf{w}}(f)$ equal to some power of
$\mathbf{x}^{\hat{\sigma}}$, which is in turn equivalent to
$\mathbf{w} \not \in V(I)$.
\end{proof}

\section{Gr\"obner complex for tropical ideals}\label{sec:complex}

\label{s:5}

In this section we define and study the Gr\"obner complex of a homogeneous tropical ideal,
and use it to show that the variety of any tropical ideal 
is always the support of a finite polyhedral complex.  
We also show that tropical ideals satisfy the 
weak Nullstellensatz.

We start by extending the definition of a polyhedral complex to
$\Rbar^{n}$. Recall that for $\sigma \subseteq \{1,\dotsc,n\}$ 
we write $\Rbar^{n}_{\sigma} =
\{\mathbf{w} \in \Rbar^{n} : w_i = \infty \text{ for } i \in \sigma,
w_i \neq \infty \text{ for } i \notin \sigma \}$.

\begin{definition}({\em Polyhedral complexes in $\Rbar^{n}$.})
\label{d:polyhedralcomplexRbar}
A {\em polyhedral complex} $\Delta$ in $\Rbar^{n}$ is a collection
of polyhedral complexes $\Delta_{\sigma} \subseteq \Rbar^{n}_{\sigma} \cong \mathbb
R^{n-|\sigma|}$ indexed by the set of all subsets $\sigma \subseteq \{1,\dots,n\}$,
with the additional requirement that if $\tau \subseteq \sigma$ and $P$ is a
polyhedron in $\Delta_{\tau}$, then the closure of $P$ in
$\Rbar^{n}$ intersected with $\Rbar^{n}_{\sigma}$ is {\em contained} in
a polyhedron of $\Delta_{\sigma}$.  The
support of any polyhedral complex is thus  a closed subset of $\Rbar^{n}$.
A polyhedral complex $\Delta$ is {\em $\mathbb R$-rational} if for all
$\sigma \subseteq \{1,\dotsc,n\}$, each polyhedron $P \in
\Delta_{\sigma}$ has the form $P= \{\mathbf{x} \in \mathbb
R^{n-|\sigma|} : A\mathbf x \leq \mathbf{b} \}$ with $\mathbf{b} \in
\mathbb R^l$ and $A \in \mathbb Q^{l \times (n-|\sigma|)}$ for some
$l \geq 0$.
\end{definition}

This definition differs slightly from the one given in 
\cite{TropicalHomology}*{\S 2.1}, in that the closure of a polyhedron
intersected with the boundary is not required to be a polyhedron in the complex, 
but just to be contained in one.  A motivation
for using this weaker condition is given in Example~\ref{e:normalcomplex} below.

The following theorem guarantees the existence of the Gr\"obner complex for 
any homogeneous tropical ideal in $\Rbar[x_0,\dots,x_n]$.  

\begin{theorem}({Gr\"obner complexes exist.}) \label{t:Grobnercomplex}
Let $I \subseteq \Rbar[x_0,\dots,x_n]$ be a homogeneous tropical
ideal.  There is a finite $\mathbb R$-rational polyhedral complex
$\Sigma(I) \subseteq \Rbar^{n+1}$, whose support is all of
$\Rbar^{n+1}$, such that for any $\sigma \subseteq \{0,\dots,n\}$ and
any $\mathbf{w}, \mathbf{w}' \in \Rbar^{n+1}_{\sigma}$, the vectors
$\mathbf{w}$ and $\mathbf{w}'$ lie in the same cell of $\Sigma(I)$ if
and only if $\inn_{\mathbf{w}}(I)=\inn_{\mathbf{w}'}(I)$.  The
polyhedral complex $\Sigma(I)$ is called the {\em Gr\"obner complex}
of $I$.
\end{theorem}

Theorem \ref{t:Grobnercomplex} is a consequence of a slightly stronger 
result, stated in Theorem~\ref{t:statepolytope}.  

\begin{example} \label{e:normalcomplex}
Consider the ideal $J := \langle xy - xz \rangle \subseteq \mathbb C[x,y,z]$, 
and let $I := \trop(J) \subseteq \Rbar[x,y,z]$.
The complex $\Sigma(I)_{\emptyset} \subseteq \mathbb R^3$ has
three cones, depending on whether $w_2-w_3$ is positive, negative, or
zero.  The intersection of the closure of the cone $\{\mathbf{w} \in \mathbb
R^3 : w_2>w_3 \}$ with $\Rbar^3_{\{1\}}$ is $\{(\infty,w_2,w_3) :
w_2>w_3 \}$, even though $\Sigma(I)_{\{1\}}$ consists of just one
polyhedron, as $\inn_{\mathbf{w}}(f) = \infty$ for all $f \in I$ and 
$\mathbf{w} \in \Rbar^3_{\{1\}}$.
\end{example}

We will make use of the following notation. 
The \emph{normal complex} $\mathcal N(f)$ of a polynomial $f \in
\Rbar[x_1,\dotsc,x_l]$ 
is the $\mathbb R$-rational polyhedral complex
in $\mathbb R^{l}$ whose polyhedra are the closures of the sets
$C[\mathbf{w}] = \{ \mathbf{w}' \in \mathbb R^{l} :
\inn_{\mathbf{w}'}(f) = \inn_{\mathbf{w}}(f) \}$ for $\mathbf{w} \in
\mathbb R^{l}$.

\begin{lemma} \label{l:complexindegreed}
Let $I \subseteq \Rbar[x_0,\dots,x_n]$ be a homogeneous tropical ideal, and 
fix a degree $d \geq 0$. There is a finite $\mathbb R$-rational polyhedral complex 
$\Sigma(I_d) \subseteq \Rbar^{n+1}$, whose support is all of $\Rbar^{n+1}$, 
such that for any $\sigma \subseteq \{0,\dots,n\}$ 
and any $\mathbf{w}, \mathbf{w}' \in \Rbar^{n+1}_{\sigma}$, the vectors
$\mathbf{w}$ and $\mathbf{w}'$ lie in the same cell of $\Sigma(I_d)$ 
if and only if $\inn_{\mathbf{w}}(I)_d=\inn_{\mathbf{w}'}(I)_d$.
\end{lemma}
\begin{proof}
Fix $\sigma \subseteq \{0,\dots,n\}$.  As in the statement of
Theorem~\ref{t:initial}, let $\mon^{\sigma}_d$ be the set of monomials in $\mon_d$
that are divisible by some variable $x_i$ with $i \in \sigma$, and write
$\M^{\sigma}_d$ for the valuated matroid $\M_d(I) / \mon^{\sigma}_d$.  
Let $F^{\sigma}_d \in \Rbar[x_0,\dots,x_n]$ be the tropical polynomial
\begin{equation} \label{eqtn:GrobnerComplexHypersurface}
 F^{\sigma}_d := \bigoplus_{B \text{ basis of } \underline{\M^{\sigma}_d}} p(B) 
\ttimes 
 \left( \prod_{\mathbf{x}^\mathbf{u} \in \mon_d \setminus (\mon^{\sigma}_d \cup B)} 
\mathbf{x}^\mathbf{u} \right). 
\end{equation}

We claim that for $\mathbf{w}, \mathbf{w}' \in \Rbar^{n+1}_{\sigma}$,
we have $\inn_{\mathbf{w}}(I)_d=\inn_{\mathbf{w}'}(I)_d$ if and only
if $\inn_{\mathbf{w}}(F_d^{\sigma})=\inn_{\mathbf{w}'}(F_d^{\sigma})$.
By replacing $I$ by the ideal $I^{\sigma} = \langle f^{\sigma} : f \in
I \rangle \subseteq \Rbar[x_i: i \notin \sigma]$ introduced in the
proof of Theorem~\ref{t:initial} we may assume that $\sigma =
\emptyset$.  Indeed, $F_d^{\sigma}(I) = F_d^{\emptyset}(I^{\sigma})$
up to tropical scaling, and $\inn_{\mathbf{w}}(I) =
\inn_{\overline{\mathbf{w}}}(I^{\sigma})$ for any $\mathbf{w} \in
\Rbar^{n+1}_{\sigma}$, where $\overline{\mathbf{w}}$ is the projection
of $\mathbf{w}$ to coordinates not in $\sigma$.

By Theorem~\ref{t:initial}, two vectors $\mathbf{w}, \mathbf{w}' \in
\Rbar_{\emptyset}^{n+1}= \mathbb R^{n+1}$ satisfy
$\inn_{\mathbf{w}}(I)_d=\inn_{\mathbf{w}'}(I)_d$ if and only if
$\inn_{\hat{\mathbf{w}}}(\M_d)=\inn_{\hat{\mathbf{w}}'}(\M_d)$,
where $\hat{\mathbf{w}} \in \mathbb R^{\mon_d}$ is given by
$\hat{\mathbf{w}}_{\mathbf{x}^{\mathbf{u}}} = \mathbf{w} \cdot
\mathbf{u}$.
Let $r = \rank(\M_d(I))$.  By
Lemma~\ref{lem:initial}, a set $B \in \binom{\mon_d}{r}$ is a basis of
$\inn_{\hat{\mathbf{w}}}(\M_d(I))$ if and only if $B$ minimizes the
function $p(B) - \sum_{\mathbf{x}^{\mathbf{u}} \in B} \mathbf{w}
\cdot \mathbf{u}$, and thus if and only if $B$ minimizes 
$p(B) + \sum_{\mathbf{x}^{\mathbf{u}} \in \mon_d \setminus B}
\mathbf{w} \cdot \mathbf{u}$.  It follows that
$\inn_{\hat{\mathbf{w}}}(\M_d(I)) =
\inn_{\hat{\mathbf{w}}'}(\M_d(I))$ if and only if the bases 
$B$ of $\underline{\M^{\sigma}_d}$ at which the minimum in
$F_d^{\emptyset}(\mathbf{w})$ is achieved are the same as for $F_d^{\emptyset}(\mathbf{w}')$,
which happens if and only if $\inn_{\mathbf{w}}(F_d^{\emptyset}) =
\inn_{\mathbf{w}'}(F_d^{\emptyset})$.

We now set the polyhedral complex $\Sigma(I_d)_{\sigma}$ in
$\Rbar^{n+1}_{\sigma}$ to be normal complex
$\mathcal N (F^{\sigma}_d)$ in $\Rbar^{n+1}_{\sigma} \cong \mathbb R^{n+1-|\sigma|}$.
It remains to check that the $\Sigma(I_d)_{\sigma}$ glue together to form 
an $\mathbb R$-rational polyhedral complex $\Sigma(I_d)$ in
$\Rbar^{n+1}$: if $\tau \subseteq \sigma$ and $P$ is a
polyhedron in $\Sigma(I_d)_{\tau}$ then the closure of $P$ in
$\Rbar^{n+1}$ intersected with $\Rbar^{n+1}_{\sigma}$ should  be contained in a 
polyhedron of $\Sigma(I_d)_{\sigma}$. Again, it suffices to consider
the case $\tau = \emptyset$. Let $C \subseteq \Rbar^{n+1}_{\emptyset} \cong \mathbb R^{n+1}$ be a
maximal open cell of $\Sigma(I_d)_\emptyset$, and suppose that
$(\mathbf v_i)_{i \geq 0}$ is a sequence of vectors in $C$ converging
to $\mathbf w \in \Rbar^{n+1}_{\sigma}$, where $\mathbf w$ also lies
in a maximal open cell of $\Sigma(I_d)_\sigma$. Since $C$ is a cell of
$\Sigma(I_d)_\emptyset$, the initial ideals $\inn_{\mathbf{v}_i}(I)$
are the same for any $i$. We will show that $\inn_{\mathbf w}(I)_d$
depends only on $\inn_{\mathbf v_i}(I)_d$, which implies the desired
containment.

By Theorem \ref{t:initial} and Lemma~\ref{lem:initial}, 
the $\Rbar$-semimodule $\inn_{\mathbf{v}_i}(I)_d$ is 
determined by the matroid whose bases are the subsets $B \in \binom{\mon_d}{r}$ 
minimizing 
\begin{equation}\label{eq:minimize}
\textstyle p(B) - \sum_{\mathbf u \in B} \mathbf v_i \cdot \mathbf u.
\end{equation}
We claim that this minimum is achieved at a single basis $B_0$.
Indeed, by Lemma~\ref{l:monomial}, there is $\mathbf{v} \in C$ with
$\inn_{\mathbf{v}}(I)$ generated by monomials.  Thus
$\inn_{\mathbf{v}'}(I)_d$ is spanned as an
$\Rbar$-semimodule by monomials for all $\mathbf{v}' \in C$, and so the corresponding
initial matroid has only one basis $B_0$.  Note that $B_0$ does not depend
on the choice of $\mathbf{v}' \in C$.  

The matroid of the $\Rbar$-semimodule
$\inn_{\mathbf w}(I)_d$ has as bases the subsets of the form $B' \sqcup
\mon^{\sigma}_d$ with $B' \in \binom{\mon_d \setminus
  \mon^{\sigma}_d}{r'}$ minimizing
\begin{equation}\label{eq:minimize2}
\textstyle p(B' \sqcup B_{\mon^{\sigma}_d}) - \sum_{\mathbf u \in B'} \mathbf w \cdot 
\mathbf u,
\end{equation} 
where $B_{\mon^{\sigma}_d}$ is any fixed basis of $\mon^{\sigma}_d$,
and $r' + |B_{\mon^{\sigma}_d}|$ is the rank of $\mathcal M_d(I)$.
  Again, this minimum is achieved at only one subset $B'_0$. We claim
  that $B'_0 = B_0 \setminus \mon^{\sigma}_d$, which implies that
  $\inn_{\mathbf w}(I)_d$ depends only on $\inn_{\mathbf v_i}(I)_d$,
  as desired.

To prove the claim, rewrite \eqref{eq:minimize} as 
\begin{equation}\label{eq:minimize3}
\textstyle p(B) - \sum_{\mathbf u \in B \setminus \mon^{\sigma}_d} \mathbf v_i \cdot \mathbf 
u - \sum_{\mathbf u \in B \cap \mon^{\sigma}_d} \mathbf v_i \cdot \mathbf u.
\end{equation}
As $i$ tends to $\infty$ the coordinates of $\mathbf v_i$ indexed by $\sigma$ 
tend to $\infty$, and the others tend to finite values. The terms $\mathbf 
v_i \cdot \mathbf u$ in \eqref{eq:minimize3} with $\mathbf u \in B \cap \mon^{\sigma}_d$ can 
thus be made arbitrarily large, while for $\mathbf u \in B \setminus \mon^{\sigma}_d$ they 
remain bounded. 
It follows that the basis $B_0$ minimizing \eqref{eq:minimize3} contains as many 
elements of $\mon^{\sigma}_d$ as possible, and so $|B_0 \cap \mon^{\sigma}_d| = r(\mon^{\sigma}_d)$. The intersection 
$B_0 \cap \mon^{\sigma}_d$ is then a basis $B_{\mon^{\sigma}_d}$ of $\mon^{\sigma}_d$. Furthermore, the set $B_0 
\setminus B_{\mon^{\sigma}_d}$ must be the subset $B'_0$ minimizing \eqref{eq:minimize2}, as 
otherwise $B'_0 \sqcup B_{\mon^{\sigma}_d}$ would yield a smaller value of 
\eqref{eq:minimize3} than $B_0$. 
\end{proof}

\begin{lemma} \label{l:Degreebound}
Let $I \subseteq \Rbar[x_0,\dots,x_n]$ be a homogeneous tropical ideal. 
There is a degree $D \geq 0$ such that for any $\mathbf{w}, \mathbf{w}' \in 
\mathbb R^{n+1}$, if $\inn_{\mathbf{w}}(I)_d
=\inn_{\mathbf{w}'}(I)_d$ for all $d \leq D$ then
$\inn_{\mathbf{w}}(I)=\inn_{\mathbf{w}'}(I)$.
\end{lemma}

\begin{proof}
Let $H_I$ be the Hilbert function of $I$.  There is only a finite
number of monomial ideals in $\Rbar[x_0,\dots,x_n]$ with Hilbert
function $H_I$, as the same is true in the polynomial ring
$K[x_0,\dots,x_n]$ for any field $K$; see, for example, \cite{MaclaganAntichains}*{Corollary 2.2}.
Let $D$ be the maximum degree of any generator of a monomial ideal
with Hilbert function $H_I$. 

Fix $e>D$, and let $\Sigma(I_e)_{\emptyset}$ be the polyhedral complex
whose existence is guaranteed by Lemma~\ref{l:complexindegreed}.  We
will show that $\Sigma(I_e)_{\emptyset}$ is refined by the polyhedral
complex $\Sigma(I_D)_{\emptyset}$.  For this, it suffices to show that
the relative interior of any maximal face in $\Sigma(I_D)_{\emptyset}$
is contained in a maximal face of $\Sigma(I_e)_{\emptyset}$. For
generic $\mathbf{v}$ in the relative interior of a maximal face in
 $\Sigma(I_e)_{\emptyset}$,
Lemma~\ref{l:monomial} implies that the initial ideal
$\inn_{\mathbf{v}}(I)$ is a monomial ideal, which has Hilbert function
$H_I$.  Thus for all $\mathbf{v}$ in this cell,
$\inn_{\mathbf{v}}(I)_e$ is generated as a $\mathbb B$-semimodule by
monomials.
Our choice of $D$ implies that 
these monomials are all
divisible by some monomial in $\inn_{\mathbf{v}}(I)_D$.  This means that 
for generic $\mathbf{v}$ and $\mathbf{v}'$, if
$\inn_{\mathbf{v}}(I)_D = \inn_{\mathbf{v}'}(I)_D$ then
$\inn_{\mathbf{v}}(I)_e = \inn_{\mathbf{v}'}(I)_e$, which proves the desired 
containment.

To conclude, fix now specific $\mathbf{w}, \mathbf{w}' \in 
\Rbar^{n+1}_{\emptyset} \cong \mathbb R^{n+1}$, and suppose that
$\inn_{\mathbf{w}}(I)_e \neq \inn_{\mathbf{w}'}(I)_e$ for some $e>D$.
This means that $\mathbf{w}$ and $\mathbf{w}'$ live in different cells
of the polyhedral complex $\Sigma(I_e)_{\emptyset}$.  As 
$\Sigma(I_e)_{\emptyset}$ is refined
by $\Sigma(I_D)_{\emptyset}$, this implies that $\inn_{\mathbf{w}}(I)_D \neq 
\inn_{\mathbf{w}'}(I)_D$, and thus the proposition follows.
\end{proof}

The following theorem is a strengthening of
Theorem~\ref{t:Grobnercomplex}. It shows that
the Gr\"obner complex of a homogeneous tropical ideal is,
in each $N(\sigma) \cong \mathbb R^{n+1-|\sigma|}$, dual to a
regular subdivision of a convex polytope.

\begin{theorem}({Gr\"obner complexes are normal complexes.})
\label{t:statepolytope}
Let $I \subseteq \Rbar[x_0,\dots,x_n]$ be a homogeneous tropical ideal. For each
$\sigma \subseteq \{0,\dots,n\}$ there is a tropical polynomial
$F^{\sigma} \in \Rbar[x_i : i \not \in \sigma]$ such that for
$\mathbf{w}, \mathbf{w}' \in \Rbar_{\sigma}^{n+1}$ we have
$\inn_{\mathbf{w}}(I)=\inn_{\mathbf{w}'}(I)$ if and only if
$\inn_{\mathbf{w}}(F^{\sigma})=\inn_{\mathbf{w}'}(F^{\sigma})$.
Moreover, the union of the normal complexes $\mathcal N(F^{\sigma}) \subseteq
\Rbar^{n+1}_{\sigma} \cong \mathbb R^{n+1-|\sigma|}$ forms an $\mathbb R$-rational polyhedral complex $\Sigma(I)$ in 
$\Rbar^{n+1}$.
\end{theorem}

\begin{proof}
For any $\sigma \subseteq \{0,\dots,n\}$ consider the ideal
$I^{\sigma} := \langle f^{\sigma} : f \in I \rangle \subseteq
\Rbar[x_i: i \notin \sigma]$ discussed in the proof of
Theorem~\ref{t:initial}.  Let $D_{\sigma}\geq 0$ be the degree given
by Lemma~\ref{l:Degreebound} for the ideal $I^{\sigma}$, and let $D:=
\max_{\sigma} D_{\sigma}$. For any $d \leq D$, let $F_d^{\sigma}$ be
the tropical polynomial given in
\eqref{eqtn:GrobnerComplexHypersurface}, and let $F^{\sigma} :=
\prod_{d \leq D} F_d^{\sigma}$.  Two vectors $\mathbf{w}, \mathbf{w}'
\in \Rbar^{n+1}_{\sigma}$ satisfy
$\inn_{\mathbf{w}}(F^{\sigma})=\inn_{\mathbf{w}'}(F^{\sigma})$ if and
only if
$\inn_{\mathbf{w}}(F_d^{\sigma})=\inn_{\mathbf{w}'}(F_d^{\sigma})$ for
all $d \leq D$. The proof of Lemma~\ref{l:complexindegreed} shows that
this is the same as
$\inn_{\overline{\mathbf{w}}}(I^\sigma)_d=\inn_{\overline{\mathbf{w}}'}(I^\sigma)_d$
for all $d \leq D$, where $\overline{\mathbf{w}}$ and
$\overline{\mathbf{w}}'$ are the projections of $\mathbf{w}$ and
$\mathbf{w}'$ to coordinates not in $\sigma$.  By
Lemma~\ref{l:Degreebound}, this is equivalent to
$\inn_{\overline{\mathbf{w}}}(I^\sigma)=\inn_{\overline{\mathbf{w}}'}(I^\sigma)$,
which is in turn the same as $\inn_{\mathbf{w}}(I) =
\inn_{\mathbf{w}'}(I)$.

Now, for any $\sigma \subseteq \{0, \dotsc, n\}$, 
the normal complex $\mathcal N(F^{\sigma}) \subseteq
\Rbar^{n+1}_{\sigma} \cong \mathbb R^{n+1-|\sigma|}$ is the $\mathbb R$-rational polyhedral complex 
obtained as the common refinement of the normal complexes 
$\mathcal N(F_d^{\sigma})$ with $d \leq D$. In 
Lemma~\ref{l:complexindegreed} we showed that for any fixed 
$d$, the complexes $\mathcal N(F_d^{\sigma})$ with $\sigma \subseteq \{0, \dotsc, 
n\}$ form a polyhedral complex in $\Rbar^{n+1}$, and thus the same is true for 
their common refinements $\mathcal N(F^{\sigma})$.
\end{proof}

\begin{example} \label{e:statepoly}
Let $J := \langle x_1-x_2, x_3-x_4 \rangle
\subseteq \mathbb C[x_1,x_2,x_3,x_4]$, and $I :=\trop(J) \subseteq
\Rbar[x_1,x_2,x_3,x_4]$. In this case, since $J$ is a linear ideal, the degree $D$ given in the
proof of Theorem \ref{t:statepolytope} is equal to $1$. We have
$F^{\emptyset} = x_1 x_3\tplus x_1  x_4 \tplus x_2  x_3 \tplus x_2  x_4$, and
$F^{\{1,2\}}=x_3 \tplus x_4$. 
Note that the Gr\"obner complex of $I$ in $\Rbar^4_{\{1,2\}}$ is not
equal to the normal complex of $F^{\emptyset}|_{x_1=x_2=\infty} =
\infty$; the polynomial $F^{\{1,2\}}$ is needed to get the comparison
between $w_3$ and $w_4$ in this stratum.
\end{example}

We now use the existence of Gr\"obner complexes to show that the
variety of a tropical ideal is an $\mathbb R$-rational polyhedral
complex.  We first prove that tropical ideals always admit finite tropical bases.

For a fan $\Sigma \subseteq N_{\mathbb R}$ and a homogeneous $f \in
\Cox(\trop(X_{\Sigma}))$,
 we set $$V_{\Sigma}(f) := (V(f) \setminus
V(\trop(B_{\Sigma})))/\ker(Q^T) \subseteq \trop(X_{\Sigma}).$$

\begin{definition}({\em Tropical bases.})
Fix a fan $\Sigma \subseteq N_{\mathbb R}$, and let $I \subseteq
\Cox(\trop(X_{\Sigma}))$ be a homogeneous ideal. A set $\{f_1, f_2,
\dotsc, f_l\} \subseteq I$ of homogeneous polynomials
forms a {\em tropical basis}
for $I$ if
 \[V_\Sigma(I) = V_\Sigma(f_1) \cap V_\Sigma(f_2) \cap \dotsb \cap 
V_\Sigma(f_l).\]
\end{definition}

\begin{theorem}\label{t:TropicalBasis}
 Any locally tropical ideal $I \subseteq \Cox(\trop(X_{\Sigma}))$ has a 
finite tropical basis.
\end{theorem}

\begin{proof}
We first show this in the case where $\Sigma$ is the fan of $\mathbb P^n$, 
so the Cox semiring is $S = \Rbar[x_0,\dots,x_n]$, 
with the stronger assumption that $I \subseteq S$ is a homogeneous tropical ideal in the
sense of Definition~\ref{d:tropicalidealprojective}.
By Proposition~\ref{p:varietynonmonomial}, the variety $V(I) \subseteq
\Rbar^{n+1}$ is the union of those cells of $\Sigma(I)$
corresponding to initial ideals not containing
a monomial.  Suppose that $C$ is a cell of $\Sigma(I)$ 
corresponding to an initial ideal $\inn_{\mathbf{w}}(I)$ that contains
a monomial $\mathbf{x}^{\mathbf{u}}$ of degree $d$.  We will construct a
homogeneous polynomial $f \in I$ with $V(f) \cap C = \emptyset$.
If $C \subseteq \Rbar^{n+1}_{\tau}$ for some nonempty $\tau
\subseteq \{0,\dots,n \}$, we can replace $I$ by $I^\tau := \langle
g^\tau : g \in I \rangle$ as in the proof of Theorem~\ref{t:initial};
if $V(f^\tau) \subseteq \Rbar^{n+1}_{\tau}$ does not intersect
$C$ then $V(f) \cap C = \emptyset$.  We may thus assume
that $C \subseteq \Rbar^{n+1}_{\emptyset} \cong \mathbb R^{n+1}$.

By definition, the initial ideal $\inn_{\mathbf{w}}(I)$ is the same
for any $\mathbf{w} \in C$.  Fix a basis $B$ of
the initial matroid $M_d(\inn_{\mathbf{w}}(I))$.  By Lemma~\ref{lem:initial} and
Theorem~\ref{t:initial}, 
$B$ is a basis of $\underline{\M_d(I)}$ as well.  Since
$\mathbf{x}^{\mathbf{u}}$ is a loop in $M_d(\inn_{\mathbf{w}}(I))$,
we have $\mathbf{x}^{\mathbf{u}} \not \in B$. 
Let $f \in I$ be the fundamental circuit
$H(B,\mathbf{x}^{\mathbf{u}})$, as in \eqref{eqtn:fundamentalcircuit}.

We claim that for any $\mathbf{w} \in C$, the initial form
$\inn_{\mathbf{w}}(f)$ is equal to $\mathbf{x}^{\mathbf{u}}$.  Suppose
not.  By Lemma~\ref{lem:initial} and Theorem~\ref{t:initial}, the
support $D$ of $\inn_{\mathbf{w}}(f)$ is a cycle in
$M_d(\inn_{\mathbf{w}}(I))$.  Note that $\mathbf{x}^{\mathbf{u}}
\in D$, as otherwise $D$ would be completely contained in the basis $B$,
and $B$ would contain a circuit of $M_d(\inn_{\mathbf{w}}(I))$. 
Since $\mathbf{x}^{\mathbf{u}}$ is a loop in
$M_d(\inn_{\mathbf{w}}(I))$, 
the vector elimination axiom in $M_d(\inn_{\mathbf{w}}(I))$ 
implies that the subset $D \setminus
\{\mathbf{x}^{\mathbf{u}}\}$ must also be a cycle in
$M_d(\inn_{\mathbf{w}}(I))$.  But then $D \setminus
\{\mathbf{x}^{\mathbf{u}}\}$ is a nonempty cycle of
$M_d(\inn_{\mathbf{w}}(I))$ contained in the basis $B$, which is a
contradiction.

We have thus shown that for any cell $C \in \Sigma(I)$ not in the
variety $V(I)$ there is a homogeneous polynomial $f \in I$ such that
$V(f) \cap C = \emptyset$. A tropical basis for $I$ can then be
obtained by taking the collection of all these polynomials
corresponding to the finitely many cells $C \in \Sigma(I)$ that
are not in $V(I)$.

We now show that if $I \subseteq \Rbar[x_1^{\pm 1},\dots,x_n^{\pm 1}]$
is a tropical ideal then there exist $f_1,\dotsc,f_l \in I$ such that
$V(f_1)\cap \dotsb \cap V(f_l) = V(I) \subseteq \mathbb R^n$. We first claim that the
homogenization $\tilde{I} \subseteq \Rbar[x_0,\dots,x_n]$ of 
$I \cap \Rbar[x_1,\dots,x_n]$ is a homogeneous tropical ideal.  
Indeed, note that for any collection of monomials $E$ in
 $\Rbar[x_1,\dots,x_n]$, the set of polynomials in $I \cap
 \Rbar[x_1,\dots,x_n]$ with support in $E$ is the same as the set of
 polynomials in $I$ with support in $E$, so $I \cap
 \Rbar[x_1,\dots,x_n]$ is a tropical ideal. Our claim follows since
 the homogenization of a tropical ideal is a homogeneous tropical
 ideal.  
Now, by the first part of the proof, there are homogeneous
polynomials $f_1,\dots, f_l \in \tilde{I}$ with
$\trop(\tilde{I}) = V(f_1)\cap \dots \cap V(f_l) \subseteq
\Rbar^{n+1}$.  Moreover, for any homogeneous $f \in \Rbar[x_0,\dots,x_n]$ we have
$V(f|_{x_0=0}) = V(f) \cap \{w_0 = 0 \}$.  Thus, under the
identification of $\mathbb R^{n}$ with $\{ \mathbf{w} \in \mathbb
R^{n+1} : w_0=0 \}$, we have $V(I) \subseteq \mathbb R^n$ equal to $V(\tilde{I}) \cap
\{ \mathbf{w} \in \mathbb R^{n+1} :w_0=0 \}$, and so $V(I) =
V(f_1|_{x_0=0}) \cap \dots \cap V(f_l|_{x_0=0})$.

We finally consider the general case that $I$ is a locally tropical ideal
in a Cox semiring $S = \Cox(\trop(X_{\Sigma})) =
\Rbar[x_1,\dots,x_m]$.  Fix $\sigma \in \Sigma$.  Then
$J:=(IS_{\mathbf{x}^{\hat{\sigma}}}/\langle x_i \sim \infty : i \in
\sigma \rangle)_{\mathbf{0}}$ is a tropical ideal in
$(S_{\mathbf{x}^{\hat{\sigma}}}/\langle x_i \sim \infty : i \in \sigma
\rangle)_{\mathbf{0}} \cong \Rbar[y_1^{\pm 1},\dots,y_p^{\pm 1}]$, where
$p = n - \dim(\sigma)$.  By the previous paragraph we can find
$f_1,\dots,f_l \in J$ such that
$V(J) \subseteq N(\sigma) \cong \mathbb R^p$ is equal to the intersection
of the hypersurfaces $V(f_1),\dots,V(f_l)$.  For each $f_i$ there is
$g_i = f_i+f_i' \in IS_{\mathbf{x}^{\hat{\sigma}}}$
with each monomial in $f_i'$ divisible by some $x_i$
with $i \in \sigma$.  There is $e \geq 0$ such that $h_i :=
(\mathbf{x}^{\hat{\sigma}})^e \ttimes g_i \in I$.  By construction we have
$V_{\sigma}(h_i) = V(f_i)$, and thus taking the collection of all $h_i$ as
$\sigma$ varies over all cones of $\Sigma$ we get a tropical basis for
$I$.
\end{proof}

We now prove the main result of this section. 
In order to state it, we extend the definition of a polyhedral complex to
the case where the ambient space is any tropical toric variety.

\begin{definition}({\em Polyhedral complexes in tropical toric varieties.})
Fix a fan $\Sigma \subseteq N_{\mathbb R} \cong \mathbb R^n$.  A {\em
  polyhedral complex} $\Delta$ in $\trop(X_{\Sigma})$ consists of a
polyhedral complex $\Delta_{\sigma}$ in each orbit $N(\sigma)
\cong \mathbb R^{n-\dim(\sigma)}$, with
the additional requirement that if $\tau$ is a face of $\sigma$ and
$P$ is a polyhedron in $\Delta_{\tau}$, then the closure of $P$ in
$\trop(X_\Sigma)$ intersected with $N(\sigma)$ is contained in a
polyhedron of $\Delta_{\sigma}$.  Note that, in particular, the
support of any polyhedral complex is a closed subset of
$\trop(X_\Sigma)$.  The polyhedral complex $\Delta$ is {\em $\mathbb
  R$-rational} if for all $\sigma \in \Sigma$, any polyhedron $P \in
\Delta_{\sigma}$ has the form $P= \{\mathbf{x} \in \mathbb
R^{n-\dim(\sigma)} : A\mathbf x \leq \mathbf{b} \}$ with $\mathbf{b}
\in \mathbb R^l$ and $A \in \mathbb Q^{l \times (n-\dim(\sigma))}$ for
some $l\geq 0$.
\end{definition}

\begin{theorem}({Tropical varieties are polyhedral complexes.})
 \label{t:polyhedralcomplex}
Fix a fan $\Sigma \subseteq N_{\mathbb R}$, and let $I \subseteq
\Cox(\trop(X_{\Sigma}))$ be a locally tropical ideal. The
variety $V_{\Sigma}(I) \subseteq
\trop(X_{\Sigma})$ is the support of a finite $\mathbb R$-rational
polyhedral complex in $\trop(X_{\Sigma})$.
\end{theorem}
\begin{proof}
By Theorem~\ref{t:TropicalBasis} we can find $f_1,\dots,f_l \in
\Cox(\trop(X_{\Sigma})) = \Rbar[x_1,\dots,x_m]$, 
homogeneous with respect to the grading on $\Cox(\trop(X_{\Sigma}))$,
such that $V_{\Sigma}(I) = V_{\Sigma}(f_1) \cap \dots \cap
V_{\Sigma}(f_l)$.  Since the intersection of two finite $\mathbb
R$-rational polyhedral complexes is a finite $\mathbb R$-rational
polyhedral complex, it suffices to show that $V_{\Sigma}(f_i)$ 
is an $\mathbb R$-rational polyhedral complex in $\trop(X_{\Sigma})$ for each $i$.  

We first prove that $V(f) \subseteq \Rbar^m$ is
a finite $\mathbb R$-rational polyhedral complex for any 
polynomial $f \in \Rbar[x_1,\dots,x_m]$.
Note that it suffices to show this for the
homogenization $\tilde{f}$ of $f$, as $V(f) \subseteq \Rbar^m$ equals
$V(\tilde{f}) \cap \{ w_0 = 0 \}$, under the identification of $\{
\mathbf{w} \in \Rbar^{m+1} : w_0 = 0 \}$ with $\Rbar^m$. 
Let $\hat{\mathcal N}(\tilde{f}) \subseteq \Rbar^{m+1}$ 
be the union over all $\sigma \subseteq \{0,\dotsc,m\}$ of the normal complexes
$\mathcal N(\tilde{f}^\sigma) \subseteq \Rbar^{m+1}_\sigma \cong \mathbb R^{m+1-|\sigma|}$, 
where $\tilde{f}^\sigma$ is obtained from $\tilde{f}$ by setting $x_i = \infty$ for all $i \in \sigma$. 
The variety $V(\tilde{f})$ is a closed union of cells in $\hat{\mathcal N}(\tilde{f})$.  
The proof that $V(\tilde{f})$ is actually a polyhedral complex proceeds as in 
the proof of Lemma~\ref{l:complexindegreed}, as this only requires
$I_d$ to determine a valuated matroid on $\mon_d$, which is the case
in degree $d=\deg(\tilde f)$ for the ideal generated by $\tilde{f}$.  

Finally, since each
$f_i$ is homogeneous with respect to the grading on
$\Cox(\trop(X_{\Sigma}))$, all $V(f_i)$ contain $\ker(Q^T)$ in their lineality spaces. As the quotient of an $\mathbb R$-rational
polyhedral complex by a rational subspace of its lineality space is
again an $\mathbb R$-rational polyhedral complex, the result follows
from Lemma~\ref{l:CoxConstructionVariety}.
\end{proof}

\begin{remark}
A corollary of Theorem~\ref{t:polyhedralcomplex} is that if
$X$ is a subvariety of the torus $T  \cong (K^*)^n$ determined by an
ideal $J \subseteq K[x_1^{\pm 1},\dots,x_n^{\pm 1}]$ then $\trop(X)$ is
the support of an $\mathbb R$-rational polyhedral complex in $\mathbb R^n$.
The proof given here is essentially a simpler version of 
the one given in \cite{TropicalBook}*{\S 2.5}. 
\end{remark}

\begin{remark}
It would be desirable to have an analogue of the full Structure
Theorem~\cite{TropicalBook}*{Theorem 3.3.5} for general tropical
ideals.  The Structure Theorem states that if a 
variety $X$ is irreducible then the polyhedral complex with support
$\trop(X)$ is pure of dimension $\dim(X)$, and carries natural
multiplicities that make it balanced.  A version for tropical ideals
will require generalizing the notion of irreducibility to this
context.
\end{remark}

The following example shows that the condition that the ideal $I$ is a
locally tropical ideal is crucial for theorems \ref{t:TropicalBasis}
and \ref{t:polyhedralcomplex}.

\begin{example}({\em Varieties of non-tropical ideals can be non-polyhedral.})
\label{e:arbitrary}
Let $\{f_\alpha\}_{\alpha \in A}$ be an arbitrary collection of tropical 
polynomials
in $\Rbar[x_1^{\pm 1}, \dotsc, x_n^{\pm 1}]$, and consider the ideal $I = \langle f_\alpha 
\rangle_{\alpha \in A}$. 
The variety $V(I) \subseteq \mathbb R^{n}$ 
satisfies $V(I) = \bigcap_{\alpha \in A} V(f_\alpha)$, which shows that any 
intersection of tropical
hypersurfaces can be the variety of an ideal in $\Rbar[x_1^{\pm 1}, \dotsc, x_n^{\pm 1}]$.
Moreover, note that any $\mathbb R$-rational half-hyperplane 
is equal to such an intersection:
If $H \subseteq \mathbb R^{n}$ is the half-hyperplane 
given by $\mathbf{a} \cdot \mathbf{x} = c$ 
and $\mathbf{b} \cdot \mathbf{x} \geq d$
with $\mathbf{a}, \mathbf{b} \in \mathbb{Z}^{n}$ and 
$c,d \in \mathbb{R}$, then 
\[H = V(\mathbf{x}^\mathbf{a} \tplus c) \cap
V(\mathbf{x}^\mathbf{b} \tplus (d-c) \ttimes \mathbf{x}^\mathbf{a}
\tplus d).\] 
Since an infinite intersection of
half-hyperplanes can result in a non-polyhedral set, 
the variety of an arbitrary ideal in $\Rbar[x_1^{\pm 1},
  \dotsc, x_n^{\pm 1}]$ need not be polyhedral.  For example, a closed
disc contained in a plane in $\mathbb R^3$ can be realized as the variety of an
ideal in $\Rbar[x_1^{\pm 1},x_2^{\pm 1}, x_3^{\pm 3}]$.
\end{example}

\begin{example} ({\em Non-realizable tropical ideals can have realizable varieties.})
\label{e:varietyofnonrealizable}
We now prove that the variety $V(I)$ defined by the
non-realizable tropical ideal $I$ of Example~\ref{ex:non-realizable}
is the standard tropical line $L$ in $\trop(\mathbb P^n)$.  This is
the tropical linear space (or Bergman fan) of the uniform matroid
$U_{2,n+1}$, which has a tropical basis consisting of all polynomials
$x_i \tplus x_j \tplus x_k$ for $0 \leq i<j<k \leq n$. Indeed, any tropical
polynomial of the form $x_i \tplus x_j \tplus x_k$ is a circuit of
$\M_1(I)$ and thus a polynomial in $I$, showing that $V(I)
\subseteq L$. For the reverse inclusion, fix $\mathbf{w} \in L$, so
$\min (w_0, w_1, \dotsc, w_n)$ is attained at least $n$ times. We may
assume that $w_0 \geq w_1 = \dotsb = w_n$. The tropical ideal $I$ is
generated by the circuits of the valuated matroids $\M_d(I)$, so it
suffices to prove that if $d \geq 0$ and $H$ is a circuit of $\M_d(I)$
then the minimum in $H(\mathbf{w})$ is achieved at least twice. 
Suppose that this minimum is achieved only once.  After
tropically scaling, the circuit $H$ has the form $H =
\bigoplus_{\mathbf{u} \in C} \mathbf{x}^{\mathbf{u}}$ with
$C$ an inclusion-minimal subset satisfying $|C| > d -
\deg(\gcd(C)) + 1$. Our assumption implies that there exist
$\mathbf{x}^{\mathbf{u}_0} \in C$ and $k \geq 0$ such that the only
monomial in $C$ not divisible by $x_0^k$ is
$\mathbf{x}^{\mathbf{u}_0}$.  Let $C' := C \setminus \{\mathbf{x}^{\mathbf{u}_0}\}$. 
The monomial $x_0 \cdot \gcd(C)$ divides
$\gcd(C')$, so $\deg(\gcd(C')) \geq \deg(\gcd(C)) + 1$.  We then have
$|C'| = |C| - 1 > d - \deg(\gcd(C)) \geq d - \deg(\gcd(C')) + 1$,
contradicting the minimality of $C$.
\end{example}

Tropical ideals satisfy the following version of the Nullstellensatz,
which is completely analogous to the classical version.  This
differs from the versions found in \cites{ShustinIzhakian,
  BertramEaston, JooMincheva}, and builds on the version in
\cite{GrigorievPodolskii}. 
The special cases when the ambient tropical toric variety is affine space
or projective space are presented in Corollary~\ref{c:affineprojectivenullstellensatz}.

\begin{theorem}({Tropical Nullstellensatz.}) \label{t:tropicalnullstellensatz}
If $I$ is a tropical ideal in 
$\Rbar[x_1^{\pm 1},\dots,x_n^{\pm 1}]$, then the variety 
$V(I) \subseteq \mathbb R^n$
is empty if and only if $I$ is the unit
ideal $\langle 0 \rangle$.

More generally, if $\Sigma$ is a simplicial fan in $N_{\mathbb R}$
and $I \subseteq \Cox(\trop(X_{\Sigma}))$ is a locally tropical ideal, then
the variety $V_\Sigma(I) \subseteq \trop(X_{\Sigma})$ is empty if and only if
$\trop(B_\Sigma)^d \subseteq I$ for some $d>0$.
\end{theorem}

\begin{proof}
The ``if'' direction of the first claim is immediate.  For the
``only-if'' direction, suppose that $I$ is a tropical ideal in
$\Rbar[x_1^{\pm 1},\dots,x_n^{\pm 1}]$ with $V(I) \subseteq \mathbb R^{n}$
empty.  Let $J := I \cap \Rbar[x_1,\dots,x_n]$, which is again a
tropical ideal. 
The variety $V(J) \cap \mathbb R^n$ is also empty.  Indeed, for any $\mathbf{w}
\in \mathbb R^n$ there exists $f \in I$ with the minimum in
$f(\mathbf{w})$ achieved only once, and then $\mathbf{x}^{\mathbf{u}}f
\in J$ for some monomial $\mathbf{x}^{\mathbf{u}}$, with the minimum
in $\mathbf{x}^{\mathbf{u}}f(\mathbf{w})$ again achieved only once.
By Theorem~\ref{t:TropicalBasis}, there exists a tropical basis $\{f_1,
f_2, \dotsc, f_l\} \subseteq J$ for $J$.  
Write $d_i$ for the maximum degree of a monomial in $f_i$.  For a
fixed degree $d \geq \max \{ d_i\}$, we consider the collection of
tropical polynomials $\{ \mathbf{x}^{\mathbf{u}} f_i : |\mathbf{u}|
\leq d-d_i \}$. 
The coefficients of these polynomials can be considered as
vectors in $\Rbar^{\mon_{\leq d}}$, where $\mon_{\leq d}$ denotes
the set of monomials of degree at most $d$.  
Let $\mathcal F_d \subseteq \Rbar^{\mon_{\leq d}}$ be the collection of these vectors.  
By \cite{GrigorievPodolskii}*{Theorem 4
  (i)}, since $f_1,\dots,f_l$ have no common solution in $\mathbb
R^n$, for $d \gg 0$ the set
\begin{align*}
\mathcal F_d^{\perp} := \{ \mathbf{y}
\in \mathbb R^{\mon_{\leq d}} : &
\text{ the minimum in } \min( a_{\mathbf{u}} + {y}_{\mathbf{u}} :
\mathbf{u} \in \mon_{\leq d} ) \text{ is achieved} \\
& \text{ at least twice for all }
\mathbf{a} = (a_{\mathbf{u}}) \in \mathcal F_d \}
\end{align*}
is empty.  Let $\mathcal J_d$ be the set of vectors in
$\Rbar^{\mon_{\leq d}}$ corresponding to polynomials in $J$ of degree
at most $d$.  As $\mathcal F_d \subseteq \mathcal J_d$, the set
$\mathcal J_d^\perp$ consisting of all $\mathbf{y} \in
\mathbb R^{\mon_{\leq d}}$ for which the minimum in $\min(
a_{\mathbf{u}} + {y}_{\mathbf{u}} : \mathbf{u} \in \mon_{\leq d})$ is
achieved at least twice for all $\mathbf{a} = (a_{\mathbf{u}}) \in
\mathcal J_d$ is also empty.  Since $J$ is a tropical ideal, $\mathcal
J_d$ is the set of vectors of a valuated matroid $\M_{\leq d}(J)$. 
The set $\mathcal J_d^{\perp}$ is the tropical linear space in $\mathbb R^{\mon_{\leq d}}$
defined by the linear tropical polynomials in $\mathcal
J_d$. It follows that $\mathcal J_d^{\perp}$ 
is empty if and only the matroid $\M_{ \leq d}(J)$ contains a loop.  
This means that  $J$ contains a monomial, and so $I = \langle 0 \rangle$.

For the general case, suppose now that $\Sigma$ is a simplicial fan,
and the ideal $I$ is a locally tropical ideal in
$S:=\Cox(\trop(X_{\Sigma}))$ with $V_{\Sigma}(I)$ empty.  For each
cone $\sigma \in \Sigma$ we then have $V_{\sigma}(I) = \emptyset$, and so by the
first part of the proof, the ideal
$(IS_{\mathbf{x}^{\hat{\sigma}}}/\langle x_i \sim \infty : i \in
\sigma \rangle )_{\mathbf{0}} = 
\langle 0 \rangle$. For all
$\sigma \in \Sigma$ there is thus a polynomial of the form $0 \tplus
f_{\sigma} \in (IS_{\mathbf{x}^{\hat{\sigma}}})_{\mathbf 0}$, where 
every term of $f_{\sigma}$ is
divisible by a variable $x_i$ with $i \in \sigma$.

We prove by induction on $\dim(\sigma)$ that $I$ contains a power of
$\mathbf{x}^{\hat{\sigma}}$.  When $\sigma$ is the
origin the polynomial $f_{\sigma}$ equals $\infty$, so $0 \in
IS_{\mathbf{x}^{\hat{\sigma}}}$, and thus some power of
$\mathbf{x}^{\hat{\sigma}}$ lies in $I$.  
Suppose now that $\dim(\sigma)>0$ and the claim is true for all $\tau$ of
smaller dimension.  By induction, for all faces $\tau$ of $\sigma$, some
power of $\mathbf{x}^{\hat{\tau}}$ lies in $I$.  Let $l$ be the
maximum such power over all $\tau \preceq \sigma$.  
It suffices to show that $f_{\sigma}$ can be chosen so that every
monomial in its support is divisible by $x_i^l$ for some $i \in
\sigma$. Indeed, since $(\mathbf{x}^{\hat{\tau}})^l \in I$ for all faces
$\tau$ of $\sigma$, as $\sigma$ is a simplicial cone 
we have $x_i^l \in IS_{\mathbf{x}^{\hat{\sigma}}}$
for all $i \in \sigma$, and so every term of $f_{\sigma}$ is in
$(IS_{\mathbf{x}^{\hat{\sigma}}})_{\mathbf{0}}$. As $0 \tplus
f_{\sigma} \in (IS_{\mathbf{x}^{\hat{\sigma}}})_{\mathbf 0}$, we can repeatedly apply
the vector elimination axiom in $(IS_{\mathbf{x}^{\hat{\sigma}}})_{\mathbf 0}$
to conclude that $0 \in
(IS_{\mathbf{x}^{\hat{\sigma}}})_{\mathbf{0}}$, and thus some power of
$\mathbf{x}^{\hat{\sigma}}$ lies in $I$.

To finish the proof we now show that we may choose $f_{\sigma}$ so
that every monomial in its support is divisible by
$x_i^l$ for some $i \in \sigma$.  Suppose that this is not possible.
For any monomial $\mathbf x^{\mathbf u}$ in $S_{\mathbf{x}^{\hat{\sigma}}}$,
let its \emph{$\sigma$-degree} $\deg_\sigma(\mathbf x^{\mathbf u})$ be its degree 
in just the variables $x_i$ with $i \in \sigma$, i.e., 
$\deg_\sigma(\mathbf x^{\mathbf u}) := \sum_{i \in \sigma} u_i$. 
Fix a choice of $f_{\sigma}$, which must then have a term not
divisible by any $x_i^l$ with $i \in \sigma$. 
We may assume that $f_{\sigma}$ has been chosen so that 
the minimum $\sigma$-degree $l'$ among all its terms is as large as possible,
and furthermore, that $f_{\sigma}$ has as few terms as possible of $\sigma$-degree equal to $l'$. 
Fix a term $a_{\mathbf{u}}\ttimes\mathbf{x}^{\mathbf{u}}$ of $f_{\sigma}$ of $\sigma$-degree $l'$. 
As every term of $f_{\sigma}$ is divisible by a variable $x_i$ with $i \in \sigma$,
all the terms of $a_{\mathbf{u}}\ttimes \mathbf{x}^{\mathbf{u}} \ttimes (0 \tplus f_{\sigma})$ except for 
$a_{\mathbf{u}}\ttimes\mathbf{x}^{\mathbf{u}}$ have $\sigma$-degree at least $l'+1$.
The vector elimination axiom in $(IS_{\mathbf{x}^{\hat{\sigma}}})_{\mathbf{0}}$ 
applied to the polynomials $0 \tplus f_{\sigma}$ and
$a_{\mathbf{u}}\ttimes \mathbf{x}^{\mathbf{u}} \ttimes (0 \tplus f_{\sigma})$
then produces a polynomial of the form $0 \tplus f'_{\sigma}$ where every term
of $f'_{\sigma}$ has $\sigma$-degree at least $l'$, and $f'_{\sigma}$
has fewer terms than $f_{\sigma}$ of $\sigma$-degree equal to $l'$. This
contradicts our choice of $f_{\sigma}$, concluding the proof. 
\end{proof}

Theorem~\ref{t:tropicalnullstellensatz} has the following immediate corollary.

\begin{corollary}\label{c:affineprojectivenullstellensatz}
If $I$ is a locally tropical ideal in $\Cox(\trop(\mathbb A^n))
=\Rbar[x_1,\dots,x_n]$ then $V(I) \subseteq \Rbar^n$ is empty if and
only if $I = \langle 0 \rangle$.

If $I$ is a homogeneous tropical ideal in $\Rbar[x_0,\dots,x_n]$ then $V(I)
\subseteq \trop(\mathbb P^n)$ is empty if and only if there exists
$d>0$ such that $\langle x_0, \dotsc, x_n \rangle^d \subseteq I$.  
\end{corollary}

\begin{remark}
The hypothesis that $I$ is a locally tropical ideal is essential in
Theorem~\ref{t:tropicalnullstellensatz}.  Consider, for example, the
ideal $I=\langle x \tplus 0, x \tplus 1 \rangle \subseteq \Rbar[x]$.
Then $V(I) \subseteq \Rbar$ is empty but $I$ is not the unit ideal;
any polynomial of the form $f \ttimes(x \tplus 0) \tplus g \ttimes(x
\tplus 1)$ involves a monomial divisible by $x$.
\end{remark}

As a simple application of the tropical Nullstellensatz, 
we classify all maximal tropical ideals of the semiring 
$\Rbar[x_1, \dotsc, x_n]$. 

\begin{example}\label{e:maximalideals}
({\em Maximal tropical ideals of $\Rbar[x_1,\dots,x_n]$.})
We show that maximal tropical ideals of the semiring
$\Rbar[x_1, \dotsc, x_n]$ are in one to one correspondence 
with points in $\Rbar^n$.
Indeed, for any $\mathbf a \in \Rbar^n$ let $J_{\mathbf a}$ be the tropical 
ideal consisting of all polynomials that tropically vanish on $\mathbf a$ 
(see Example \ref{e:affineidealofpoint}).
If a tropical ideal $I$ satisfies $V(I) \neq \emptyset$ then $I$ must be contained 
in one of the $J_{\mathbf a}$. On the other hand, if $V(I) = \emptyset$ then, by the 
tropical Nullstellensatz, $I$ must be the unit ideal $\langle 0 \rangle$.
It follows that the $J_{\mathbf a}$ are the only maximal tropical ideals of 
$\Rbar[x_1, \dotsc, x_n]$.
\end{example}

\end{document}